\documentclass[10pt, a4paper]{amsart}
\usepackage{amscd,amsmath,amssymb,amsfonts,amsthm,ascmac}
\usepackage{enumerate,mathrsfs,stmaryrd,latexsym, comment, mathtools, mathdots} 
\usepackage[all]{xy}
\usepackage[top=27truemm,bottom=21truemm,left=26truemm,right=26truemm]{geometry}
\usepackage{graphicx}
\def\r{\mathbb{R}}
\def\c{\mathbb{C}}
\def\q{\mathbb{Q}}
\def\z{\mathbb{Z}}

\newtheorem{thm}{Theorem}[section]
\newtheorem{defi}[thm]{Definition}
\newtheorem{rem}[thm]{Remark}
\newtheorem{prop}[thm]{Proposition}
\newtheorem{ex}[thm]{Example}
\newtheorem{cor}[thm]{Corollary}
\newtheorem{lem}[thm]{Lemma}

\newtheorem*{ack}{Acknowledgments}

\newtheorem*{theor1}{Theorem 1}
\newtheorem*{theor2}{Theorem 2}
\newtheorem*{exam3}{Example 3}
\newtheorem*{coro4}{Corollary 4}

 \makeatletter
    
    \@addtoreset{equation}{section}
  \makeatother
\title[Polyhedral realizations of crystal bases and convex-geometric Demazure operators]{\fontsize{11pt}{11pt}\selectfont Polyhedral realizations of crystal bases and convex-geometric Demazure operators}
\date{\today}
\author[N. Fujita]{\fontsize{10pt}{10pt}\selectfont Naoki Fujita}
\begin{document}
\address{Department of Mathematics, Tokyo Institute of Technology, 2-12-1 Oh-okayama, Meguro-ku, Tokyo 152-8551, Japan}
\email{fujita.n.ac@m.titech.ac.jp}
\thanks{The work was partially supported by Grant-in-Aid for JSPS Fellows (No.\ 16J00420)}
\begin{abstract}
The main object in this paper is a certain rational convex polytope whose lattice points give a polyhedral realization of a highest weight crystal basis. This is also identical to a Newton-Okounkov body of a flag variety, and it gives a toric degeneration. In this paper, we prove that a specific class of this polytope is given by Kiritchenko's Demazure operators on polytopes. This implies that polytopes in this class are all lattice polytopes. As an application, we give a sufficient condition for the corresponding toric variety to be Gorenstein Fano.
\end{abstract}
\maketitle
\setcounter{tocdepth}{1}
\tableofcontents
\section{Introduction}

The theory of crystal bases \cite{Kas1, Kas2} gives a combinatorial skeleton of a representation of a semisimple Lie algebra. In the theory of crystal bases, it is important to give their concrete realizations. Until now, many useful realizations have been discovered; Nakashima-Zelevinsky's polyhedral realization \cite{Nak1, NZ} is one of them, which realizes a highest weight crystal basis as the set of lattice points in some rational convex polytope. This polytope is called a Nakashima-Zelevinsky polytope. The author and Naito \cite{FN} proved that the Nakashima-Zelevinsky polytope is identical to a Newton-Okounkov body of a flag variety. The theory of Newton-Okounkov bodies was introduced by Okounkov \cite{Oko1, Oko2, Oko3}, and afterward developed independently by Kaveh-Khovanskii \cite{KK1, KK2} and by Lazarsfeld-Mustata \cite{LM}. A remarkable fact is that the theory of Newton-Okounkov bodies gives a systematic method of constructing toric degenerations \cite[Theorem 1]{And}; in particular, there exists a flat degeneration of the flag variety to the normal toric variety associated with the Nakashima-Zelevinsky polytope. In this paper, we relate Nakashima-Zelevinsky polytopes with Demazure operators on polytopes.

To be more precise, let $\mathfrak{g}$ be a semisimple Lie algebra, $P_+$ the set of dominant integral weights, $I = \{1, \ldots, n\}$ an index set for the vertices of the Dynkin diagram, and $\{\alpha_i \mid i \in I\}$ the set of simple roots. For $\lambda \in P_+$, we denote by $V(\lambda)$ the irreducible highest weight $\mathfrak{g}$-module with highest weight $\lambda$, and by $\mathcal{B}(\lambda)$ the crystal basis for $V(\lambda)$. Fix a reduced word ${\bf i} = (i_1, \ldots, i_N) \in I^N$ for the longest element $w_0$ in the Weyl group. We associate to ${\bf i}$ a specific parametrization $\Psi_{\bf i} \colon \mathcal{B}(\lambda) \hookrightarrow \z^N$ of $\mathcal{B}(\lambda)$, which gives an explicit description of the crystal structure; see Section 3 for the precise definition. Nakashima-Zelevinsky \cite{NZ} and Nakashima \cite{Nak1} described explicitly the image $\Psi_{\bf i} (\mathcal{B}(\lambda))$ under some technical assumptions on ${\bf i}$. The author and Naito \cite{FN} proved that the image $\Psi_{\bf i} (\mathcal{B}(\lambda))$ is identical to the set of lattice points in some rational convex polytope $\Delta_{\bf i} (\lambda)$ without any assumptions on ${\bf i}$. We call $\Delta_{\bf i} (\lambda)$ the Nakashima-Zelevinsky polytope associated with ${\bf i}$ and $\lambda$.

The theory of Demazure operators on polytopes was introduced by Kiritchenko \cite{Kir1} to construct a (possibly virtual) convex polytope, whose lattice points yield the character of $V(\lambda)$. For instance, Gelfand-Zetlin polytopes \cite{GZ} and Grossberg-Karshon's twisted cubes \cite{GK} are obtained in a uniform way (see \cite{Kir1}). For $i \in I$ and $1 \le k \le N$ with $i_k = i$, let $D_i ^{(k)}$ denote the corresponding Demazure operator on polytopes; see Section 2 for the precise definition. This operator is defined for a specific class of polytopes, called parapolytopes. Our purpose is to compute $D_{i_N} ^{(N)} \cdots D_{i_1} ^{(1)} ({\bf a})$ for specific ${\bf a} \in \r^N$. Note that $D_{i_N} ^{(N)} \cdots D_{i_1} ^{(1)} ({\bf a})$ is not necessarily well-defined as we will see in Example \ref{e:counter example}. For $i \in I$, we denote by $d_i$ the number of $1 \le k \le N$ such that $i_k = i$. For $\lambda \in P_+$, we write $\lambda = \sum_{i \in I} \hat{\lambda}_i d_i \alpha_i$, and set 
\begin{align*}
{\bf a}_\lambda \coloneqq -\Psi_{\bf i}(b_{w_0 \lambda}) + (\hat{\lambda}_{i_1}, \ldots, \hat{\lambda}_{i_N}), 
\end{align*}
where $b_{w_0 \lambda} \in \mathcal{B}(\lambda)$ is the lowest weight element. For subsets $X, Y \subset \r^N$, we define $X + Y$ to be the Minkowski sum: \[X + Y \coloneqq \{x + y \mid x \in X,\ y \in Y\}.\] The following are the main results of this paper.

\vspace{2mm}\begin{theor1}[{Theorem \ref{t:main result}}]
Let ${\bf i} = (i_1, \ldots, i_N) \in I^N$ be a reduced word for $w_0$, and $\lambda \in P_+$. Assume that the Nakashima-Zelevinsky polytope $\Delta_{\bf i}(\lambda)$ is a parapolytope.
\begin{enumerate}
\item[{\rm (1)}] The polytope $\Delta_{\bf i}(\lambda)$ is a lattice polytope.
\item[{\rm (2)}] The polytope $D_{i_N} ^{(N)} \cdots D_{i_1} ^{(1)} ({\bf a}_\lambda)$ is well-defined.
\item[{\rm (3)}] The following equality holds$:$ \[D_{i_N} ^{(N)} \cdots D_{i_1} ^{(1)} ({\bf a}_\lambda) = -\Delta_{\bf i}(\lambda) + (\hat{\lambda}_{i_1}, \ldots, \hat{\lambda}_{i_N}).\] 
\end{enumerate}
\end{theor1}

\vspace{2mm}\begin{theor2}[{Theorem \ref{t:corollary 1}}]
Let ${\bf i} \in I^N$ be a reduced word for $w_0$, and $\lambda, \mu \in P_+$. Assume that the polytopes $\Delta_{\bf i}(\lambda), \Delta_{\bf i}(\mu)$, and $\Delta_{\bf i}(\lambda + \mu)$ are all parapolytopes. Then, the following equalities hold$:$ 
\begin{align*}
&\Psi_{\bf i}(\mathcal{B}(\lambda + \mu)) = \Psi_{\bf i}(\mathcal{B}(\lambda)) + \Psi_{\bf i}(\mathcal{B}(\mu)),\ {\it and}\\
&\Delta_{\bf i}(\lambda + \mu) = \Delta_{\bf i}(\lambda) + \Delta_{\bf i}(\mu).
\end{align*}
\end{theor2}\vspace{2mm}

We give some examples of $\Delta_{\bf i}(\lambda)$ which are parapolytopes. 

\vspace{2mm}\begin{exam3}[{Examples \ref{e:main type A}, \ref{e:main type BCD}, \ref{e:main type G}}]\normalfont
The Nakashima-Zelevinsky polytope $\Delta_{\bf i} (\lambda)$ is a parapolytope for all $\lambda \in P_+$ if
\begin{enumerate}
\item[{\rm (i)}] $\mathfrak{g}$ is of type $A_n$, and ${\bf i} = (1, 2, 1, 3, 2, 1, \ldots, n, n-1, \ldots, 1)$;
\item[{\rm (ii)}] $\mathfrak{g}$ is of type $B_n$ or $C_n$, and ${\bf i} = (n, n-1, \ldots, 1, n, n-1, \ldots, 1, \ldots, n, n-1, \ldots, 1) \in I^{n^2}$;
\item[{\rm (iii)}] $\mathfrak{g}$ is of type $D_n$, and ${\bf i} = (n, n-1, \ldots, 1, n, n-1, \ldots, 1, \ldots, n, n-1, \ldots, 1) \in I^{n (n-1)}$;
\item[{\rm (iv)}] $\mathfrak{g}$ is of type $G_2$, and ${\bf i} = (1, 2, 1, 2, 1, 2)$ or ${\bf i} = (2, 1, 2, 1, 2, 1)$.
\end{enumerate}
\end{exam3}\vspace{2mm}

Let $G/B$ be the full flag variety associated with $\mathfrak{g}$, and $X(\Delta_{\bf i}(\lambda))$ the normal toric variety associated with the rational convex polytope $\Delta_{\bf i}(\lambda)$. Then, we obtain a flat degeneration of $G/B$ to $X(\Delta_{\bf i}(\lambda))$ by the theory of Newton-Okounkov bodies \cite{And}; such a degeneration to a toric variety is called a toric degeneration. Toric degenerations of $G/B$ have been studied from various points of view such as standard monomial theory \cite{Chi, GL}, string parametrizations of dual canonical bases \cite{AB, Cal}, Newton-Okounkov bodies \cite{FaFL, FeFL, Kav, Kir}, and so on; see \cite{FaFL2} for a survey on this topic. Let $P_{++} \subset P_+$ denote the set of regular dominant integral weights. In this paper, we apply Alexeev-Brion's argument \cite{AB} to $\Delta_{\bf i}(\lambda)$, which implies that the toric varieties $X(\Delta_{\bf i}(\lambda))$, $\lambda \in P_{++}$, are all identical and Gorenstein Fano if 
\begin{enumerate}
\item[{\rm (i)}] $\Delta_{\bf i}(\lambda + \mu) = \Delta_{\bf i}(\lambda) + \Delta_{\bf i}(\mu)$ for all $\lambda, \mu \in P_+$;
\item[{\rm (ii)}] the polytope $\Delta_{\bf i}(2\rho)$ is a lattice polytope,
\end{enumerate}
where $\rho$ is the half sum of the positive roots. Hence we obtain the following by Theorems 1, 2.

\vspace{2mm}\begin{coro4}
Take $\mathfrak{g}$ and ${\bf i}$ as in Example {\rm 3}. Then, the toric varieties $X(\Delta_{\bf i}(\lambda))$, $\lambda \in P_{++}$, are all identical and Gorenstein Fano.
\end{coro4}\vspace{2mm}

If $\mathfrak{g}$ is of type $A_n$, and ${\bf i} = (1, 2, 1, 3, 2, 1, \ldots, n, n-1, \ldots, 1)$, then the Nakashima-Zelevinsky polytope $\Delta_{\bf i} (\lambda)$ is identical to the corresponding Gelfand-Zetlin polytope (see Example \ref{e:GZ}). Hence in this case, Theorems 1, 2 and Corollary 4 are not new (see \cite{AB, Kir1}).

In addition, we mention that a relation between convex-geometric Demazure operators and the additivity with respect to the Minkowski sum is discussed in \cite{Kir3}.

This paper is organized as follows. In Section 2, we recall the definition of Kiritchenko's Demazure operators on polytopes. In Section 3, we review some basic facts about crystal bases and their polyhedral realizations. In Section 4, we prove Theorems 1, 2 above. In Section 5, we study the crystal structure on the set of lattice points in $\Delta_{\bf i} (\lambda)$. Section 6 is devoted to some applications to toric varieties associated with Nakashima-Zelevinsky polytopes; in particular, we show Corollary 4 above.

\vspace{2mm}\begin{ack}\normalfont
The author is greatly indebted to Satoshi Naito for numerous helpful suggestions and fruitful discussions. The author would also like to express his gratitude to Dave Anderson and Valentina Kiritchenko for useful comments and suggestions. At the conference ``Algebraic Analysis and Representation Theory'' in June 2017, the author gave a poster presentation on the result of this paper. But there was a gap in the proof at that time, and the condition of the main result has been corrected from the one at the conference.
\end{ack}\vspace{2mm}

\section{Convex-geometric Demazure operators}

Let $G$ be a connected, simply-connected semisimple algebraic group over $\mathbb{C}$, $\mathfrak{g}$ its Lie algebra, $W$ the Weyl group, $I = \{1, \ldots, n\}$ an index set for the vertices of the Dynkin diagram, and $(c_{i, j})_{i, j \in I}$ the Cartan matrix. We fix a reduced word ${\bf i} = (i_1, \ldots, i_N) \in I^N$ for the longest element $w_0 \in W$. For $i \in I$, let $d_i$ denote the number of $1 \le k \le N$ such that $i_k = i$. We identify $\r^N$ with the direct sum $\r^{d_1} \oplus \cdots \oplus \r^{d_n}$ as follows:
\begin{align*}
\r^N &\xrightarrow{\sim} \r^{d_1} \oplus \cdots \oplus \r^{d_n},\\
(a_1, \ldots, a_N) &\mapsto (a_1 ^{(1)}, \ldots, a_{d_1} ^{(1)}, \ldots, a_1 ^{(n)}, \ldots, a_{d_n} ^{(n)}),
\end{align*}
where we set $(a_1 ^{(i)}, \ldots, a_{d_i} ^{(i)}) \coloneqq (a_k)_{1 \le k \le N;\ i_k = i}$. If we define an $\r$-linear subspace $(\r^{d_i})^\perp \subset \r^N$ to be \[(\r^{d_i})^\perp \coloneqq \bigoplus_{1 \le j \le n;\ j \neq i} \r^{d_j},\] then we have $\r^N = (\r^{d_i})^\perp \oplus \r^{d_i}$. A subset $P \subset \r^N$ is called a \emph{convex polytope} if it is the convex hull of a finite number of points. Let $\mathscr{P}_N$ denote the set of convex polytopes in $\r^N$. This set is endowed with a commutative semigroup structure by the Minkowski sum of convex polytopes: \[P_1 + P_2 \coloneqq \{p_1 + p_2 \mid p_1 \in P_1,\ p_2 \in P_2\}.\] For $c \in \r_{\ge 0}$ and a convex polytope $P \subset \r^N$, define a convex polytope $c P \subset \r^N$ by $c P \coloneqq \{c p \mid p \in P\}$. We denote by $F(\r^N)$ the set of $\r$-valued functions on $\r^N$. For a convex polytope $P \subset \r^N$, let $\mathbb{I}_P \in F(\r^N)$ be the characteristic function of $P$, that is, 
\begin{align*}
\mathbb{I}_P (x) =
\begin{cases}
1 &{\rm if}\ x \in P,\\
0 &{\rm otherwise}.
\end{cases}
\end{align*}

\vspace{2mm}\begin{defi}[{\cite[Definition 2]{Kir1}}]\normalfont
A convex polytope $P \subset \r^N$ is called a \emph{parapolytope} if for all $i \in I$ and ${\bf c} \in \r^N$, there exist $\mu = (\mu_1, \ldots, \mu_{d_i}),\ \nu = (\nu_1, \ldots, \nu_{d_i}) \in \r^{d_i}$ such that \[P \cap ({\bf c} + \r^{d_i}) = {\bf c} + \Pi(\mu, \nu),\] where $[\mu_k, \nu_k] \coloneqq \{x \in \r \mid \mu_k \le x \le \nu_k\} \subset \r$ for $1 \le k \le d_i$, and \[\Pi(\mu, \nu) \coloneqq [\mu_1, \nu_1] \times \cdots \times [\mu_{d_i}, \nu_{d_i}] \subset \r^{d_i}.\]
\end{defi}\vspace{2mm}

Let $\mathscr{P}_\Box \subset \mathscr{P}_N$ denote the set of parapolytopes in $\r^N$. For $1 \le k \le N$, we set \[\mathscr{P}_\Box (k) \coloneqq \{P \in \mathscr{P}_\Box \mid {\rm the\ coordinate\ function}\ a_k\ {\rm is\ constant\ on}\ P\}.\] For $i \in I$, define an $\r$-linear function $l_i \colon \r^N \rightarrow \r$ by \[l_i({\bf a}) \coloneqq -\sum_{j \in I;\ j \neq i} c_{i, j} (a^{(j)} _1 + \cdots + a^{(j)} _{d_j}).\] Following \cite[Sect.\ 2.3]{Kir1}, we define a \emph{convex-geometric Demazure operator} $D_i ^{(k)} \colon \mathscr{P}_\Box (k) \rightarrow F(\r^N)$ for $i \in I$ and $1 \le k \le N$ such that $i_k = i$ as follows. We take $P \in \mathscr{P}_\Box (k)$, and denote by $1 \le m_k \le d_i$ the number of $1 \le l \le k$ such that $i_l = i_k$. 

First, we consider the case $P \subset {\bf c} + \r^{d_i}$ for some ${\bf c} \in (\r^{d_i})^\perp$. Write \[P = {\bf c} + \Pi(\mu, \nu) = {\bf c} + [\mu_1, \nu_1] \times \cdots \times [\mu_{d_i}, \nu_{d_i}],\] and set \[\nu_{m_k} ^\prime \coloneqq \nu_{m_k} + l_i ({\bf c}) -\sum_{1 \le l \le d_i} (\mu_l + \nu_l).\] We define $\nu^\prime \in \r^{d_i}$ (resp., $\mu^\prime \in \r^{d_i}$) by replacing $\nu_{m_k}$ in $\nu$ (resp., $\mu_{m_k}$ in $\mu$) by $\nu_{m_k} ^\prime$. If $\nu_{m_k} ^\prime \ge \nu_{m_k}$, then we set \[D_i ^{(k)}(P) \coloneqq \mathbb{I}_{{\bf c} + \Pi(\mu, \nu^\prime)}.\] If $\nu_{m_k} ^\prime < \nu_{m_k}$, then we set \[D_i ^{(k)}(P) \coloneqq -\mathbb{I}_{{\bf c} + \Pi(\mu^\prime, \nu)} + \mathbb{I}_P + \mathbb{I}_{P^\prime},\] where $P^\prime$ is the facet of ${\bf c} + \Pi(\mu^\prime, \nu)$ parallel to $P$.

In general, we define $D_i ^{(k)} (P) \in F(\r^N)$ by \[D_i ^{(k)}(P)|_{{\bf c} + \r^{d_i}} \coloneqq D_i ^{(k)}(P \cap ({\bf c} + \r^{d_i}))\] for ${\bf c} \in (\r^{d_i})^\perp$. 

\vspace{2mm}\begin{defi}\normalfont
Let $1 \le k \le N$, $i \coloneqq i_k$, and $P \in \mathscr{P}_\Box (k)$. If the function $D_i ^{(k)}(P)$ is identical to the characteristic function $\mathbb{I}_Q$ of a convex polytope $Q$, then by abuse of notation, we write $Q = D_i ^{(k)}(P)$.
\end{defi}

\vspace{2mm}\begin{rem}\normalfont
In the paper \cite{Kir1}, she defined convex-geometric Demazure operators for \emph{convex parachains}. Even for parapolytopes, our definition of convex-geometric Demazure operators is slightly different from hers since we specify which direction we expand in.
\end{rem}\vspace{2mm}

See \cite[Sect.\ 2.4]{Kir1} for examples of functions constructed by convex-geometric Demazure operators. Our purpose is to compute $D_{i_N} ^{(N)} \cdots D_{i_1} ^{(1)} ({\bf a})$ for specific ${\bf a} \in \r^N$. Note that $D_{i_N} ^{(N)} \cdots D_{i_1} ^{(1)} ({\bf a})$ is not necessarily well-defined as the following example.

\vspace{2mm}\begin{ex}\normalfont\label{e:counter example}
Let $G = SL_4(\c)$, and ${\bf i} = (2, 1, 2, 3, 2, 1) \in I^6$, which is a reduced word for $w_0$. Then, the functions $l_i$, $i \in I$, are given by \[l_1({\bf a}) = l_3({\bf a}) = a_1 ^{(2)} + a_2 ^{(2)} + a_3 ^{(2)}\ {\rm and}\ l_2({\bf a}) = a_1 ^{(1)} + a_2 ^{(1)} + a_1 ^{(3)}\] for ${\bf a} = (a_1 ^{(1)}, a_2 ^{(1)}, a_1 ^{(2)}, a_2 ^{(2)}, a_3 ^{(2)}, a_1 ^{(3)}) \in \r^6 = \r^2 \oplus \r^3 \oplus \r$. If we set \[{\bf a}_{\rm low} \coloneqq -\left(\frac{5}{4}, \frac{1}{4}, \frac{1}{3}, \frac{1}{3}, \frac{4}{3}, \frac{3}{2}\right) \in \r^2 \oplus \r^3 \oplus \r,\] then we have $D_2 ^{(1)}({\bf a}_{\rm low}), D_1 ^{(2)} D_2 ^{(1)} ({\bf a}_{\rm low}), D_2 ^{(3)} D_1 ^{(2)} D_2 ^{(1)} ({\bf a}_{\rm low}) \in \mathscr{P}_\Box$ and $D_3 ^{(4)} D_2 ^{(3)} D_1 ^{(2)} D_2 ^{(1)} ({\bf a}_{\rm low}) \in \mathscr{P}_6$. In addition, the polytope $D_3 ^{(4)} D_2 ^{(3)} D_1 ^{(2)} D_2 ^{(1)} ({\bf a}_{\rm low})$ is given by the following conditions: 
\begin{align*}
&(a^{(1)} _2, a^{(2)} _3) = \left(-\frac{1}{4}, -\frac{4}{3}\right),\ -\frac{1}{3} \le a^{(2)} _1 \le \frac{2}{3},\ -\frac{5}{4} \le a^{(1)} _1 \le a^{(2)} _1 + \frac{1}{12},\\
&-\frac{1}{3} \le a^{(2)} _2 \le \min\left\{a^{(1)} _1 + \frac{11}{12}, \frac{2}{3}\right\},\ -\frac{3}{2} \le a^{(3)} _1 \le a^{(2)} _1 + a^{(2)} _2 + \frac{1}{6}.
\end{align*}
Hence for ${\bf c} \coloneqq (-\frac{1}{4}, -\frac{1}{4}, 0, 0, 0, \frac{1}{2}) \in (\r^{d_2})^\perp$, the intersection $D_3 ^{(4)} D_2 ^{(3)} D_1 ^{(2)} D_2 ^{(1)} ({\bf a}_{\rm low}) \cap ({\bf c} + \r^{d_2})$ is identified with the set of $(a^{(2)} _1, a^{(2)} _2, a^{(2)} _3) \in \r^3$ satisfying the following conditions: \[-\frac{1}{3} \le a^{(2)} _1 \le \frac{2}{3},\ -a^{(2)} _1 + \frac{1}{3} \le a^{(2)} _2 \le \frac{2}{3},\ a^{(2)} _3 = -\frac{4}{3}.\] Since this is not of the form $\Pi(\mu, \nu)$, we deduce that $D_3 ^{(4)} D_2 ^{(3)} D_1 ^{(2)} D_2 ^{(1)}({\bf a}_{\rm low})$ is not a parapolytope, and hence that $D_2 ^{(5)} D_3 ^{(4)} D_2 ^{(3)} D_1 ^{(2)} D_2 ^{(1)} ({\bf a}_{\rm low})$ is not well-defined.
\end{ex}\vspace{2mm}

\section{Polyhedral realizations of crystal bases}

In this section, we review some fundamental properties of polyhedral realizations of crystal bases, following \cite{FN, Nak1, NZ}. We start with recalling the definition of abstract crystals, introduced in \cite{Kas4}. Choose a Borel subgroup $B \subset G$ and a maximal torus $T \subset B$. Denote by $\mathfrak{t}$ the Lie algebra of $T$, by $\mathfrak{t}^\ast \coloneqq {\rm Hom}_\c (\mathfrak{t}, \c)$ its dual space, and by $\langle \cdot, \cdot \rangle \colon \mathfrak{t}^\ast \times \mathfrak{t} \rightarrow \c$ the canonical pairing. Let $\{\alpha_i \mid i \in I\} \subset \mathfrak{t}^\ast$ be the set of simple roots, $\{h_i \mid i \in I\} \subset \mathfrak{t}$ the set of simple coroots, and $P \subset \mathfrak{t}^\ast$ the weight lattice.

\vspace{2mm}\begin{defi}[{\cite[Definition 1.2.1]{Kas4}}]\normalfont
A \emph{crystal} $\mathcal{B}$ is a set equipped with maps 
\begin{enumerate}
\item[] ${\rm wt} \colon \mathcal{B} \rightarrow P$,
\item[] $\varepsilon_i \colon \mathcal{B} \rightarrow \z \cup \{-\infty\}$, $\varphi_i \colon \mathcal{B} \rightarrow \z \cup \{-\infty\}\ {\rm for}\ i \in I$,\ and
\item[] $\tilde{e}_i \colon \mathcal{B} \rightarrow \mathcal{B} \cup \{0\}$, $\tilde{f}_i \colon \mathcal{B} \rightarrow \mathcal{B} \cup \{0\}\ {\rm for}\ i \in I$,
\end{enumerate}
satisfying the following conditions:
\begin{enumerate}
\item[(i)] $\varphi_i(b) = \varepsilon_i(b) + \langle{\rm wt}(b), h_i\rangle$ for $i \in I$,
\item[(ii)] ${\rm wt}(\tilde{e}_i b) = {\rm wt}(b) + \alpha_i$, $\varepsilon_i(\tilde{e}_i b) = \varepsilon_i(b) -1$, and $\varphi_i(\tilde{e}_i b) = \varphi_i(b) +1$ for $i \in I$ and $b \in \mathcal{B}$ such that $\tilde{e}_i b \in \mathcal{B}$,
\item[(iii)] ${\rm wt}(\tilde{f}_i b) = {\rm wt}(b) - \alpha_i$, $\varepsilon_i(\tilde{f}_i b) = \varepsilon_i(b) +1$, and $\varphi_i(\tilde{f}_i b) = \varphi_i(b) -1$ for $i \in I$ and $b \in \mathcal{B}$ such that $\tilde{f}_i b \in \mathcal{B}$,
\item[(iv)] $b^\prime = \tilde{e}_i b$ if and only if $b = \tilde{f}_i b^\prime$ for $i \in I$ and $b, b^\prime \in \mathcal{B}$,
\item[(v)] $\tilde{e}_i b = \tilde{f}_i b = 0$ for $i \in I$ and $b \in \mathcal{B}$ such that $\varphi_i(b) = -\infty$;
\end{enumerate}
here, $- \infty$ and $0$ are additional elements that are not contained in $\z$ and $\mathcal{B}$, respectively. 
\end{defi}\vspace{2mm}

The maps $\tilde{e}_i$ and $\tilde{f}_i$ are called the \emph{Kashiwara operators}.

\vspace{2mm}\begin{ex}\normalfont\label{example of crystals}
For $\lambda \in P$, let $R_\lambda = \{r_\lambda\}$ be a crystal consisting of only one element, given by: ${\rm wt}(r_\lambda) = \lambda$, $\varepsilon_i(r_\lambda) = - \langle\lambda, h_i\rangle$, $\varphi_i(r_\lambda) = 0$, and $\tilde{e}_i r_\lambda = \tilde{f}_i r_\lambda = 0$.
\end{ex}\vspace{2mm}

\begin{ex}\normalfont
For $i \in I$, we define a crystal $\widetilde{\mathcal{B}}_i \coloneqq \{(x)_i \mid x \in \z\}$ as follows:
\begin{align*}
&{\rm wt}((x)_i) \coloneqq -x \alpha_i,\ \varepsilon_i((x)_i) \coloneqq x,\ \varphi_i((x)_i) \coloneqq -x,\ \tilde{e}_i(x)_i \coloneqq (x-1)_i,\ \tilde{f}_i(x)_i \coloneqq (x+1)_i,\ {\rm and}\\
&\varepsilon_j((x)_i) = \varphi_j((x)_i) \coloneqq -\infty,\ \tilde{e}_j(x)_i = \tilde{f}_j(x)_i \coloneqq 0\ {\rm for}\  j \neq i.
\end{align*}
\end{ex}

\vspace{2mm}\begin{defi}[{\cite[Sect.\ 1.2]{Kas4}}]\normalfont
Let $\mathcal{B}_1, \mathcal{B}_2$ be two crystals. A map \[\psi \colon \mathcal{B}_1 \cup \{0\} \rightarrow \mathcal{B}_2 \cup \{0\}\] is called a \emph{strict morphism} of crystals from $\mathcal{B}_1$ to $\mathcal{B}_2$ if it satisfies the following conditions:
\begin{enumerate}
\item[(i)] $\psi(0) = 0$,
\item[(ii)] ${\rm wt}(\psi(b)) = {\rm wt}(b)$, $\varepsilon_i(\psi(b)) = \varepsilon_i(b)$, and $\varphi_i(\psi(b)) = \varphi_i(b)$ for $i \in I$ and $b \in \mathcal{B}_1$ such that $\psi(b) \in \mathcal{B}_2$,
\item[(iii)] $\tilde{e}_i \psi(b) = \psi(\tilde{e}_i b)$ and $\tilde{f}_i \psi(b) = \psi(\tilde{f}_i b)$ for $i \in I$ and $b \in \mathcal{B}_1$;
\end{enumerate}
here, if $\psi(b) = 0$, then we set $\tilde{e}_i \psi(b) = \tilde{f}_i \psi(b) = 0$. An injective strict morphism is called a \emph{strict embedding} of crystals. 
\end{defi}\vspace{2mm}

Consider the total order $<$ on $\z \cup \{-\infty\}$ given by the usual order on $\z$, and by $-\infty < s$ for all $s \in \z$. For two crystals $\mathcal{B}_1, \mathcal{B}_2$, we can define another crystal $\mathcal{B}_1 \otimes \mathcal{B}_2$, called the \emph{tensor product} of $\mathcal{B}_1$ and $\mathcal{B}_2$, as follows (see \cite[Sect.\ 1.3]{Kas4}): 
\begin{align*}
&\mathcal{B}_1 \otimes \mathcal{B}_2 \coloneqq \{b_1 \otimes b_2 \mid b_1 \in \mathcal{B}_1,\ b_2 \in \mathcal{B}_2\},\\
&{\rm wt}(b_1 \otimes b_2) \coloneqq {\rm wt}(b_1) + {\rm wt}(b_2),\\
&\varepsilon_i(b_1 \otimes b_2) \coloneqq \max\{\varepsilon_i(b_1),\ \varepsilon_i(b_2) - \langle{\rm wt}(b_1), h_i\rangle\},\\
&\varphi_i(b_1 \otimes b_2) \coloneqq \max\{\varphi_i(b_2),\ \varphi_i(b_1) + \langle{\rm wt}(b_2), h_i\rangle\},\\
&\tilde{e}_i(b_1 \otimes b_2) \coloneqq
\begin{cases}
\tilde{e}_i b_1 \otimes b_2 &{\rm if}\ \varphi_i (b_1) \ge \varepsilon_i (b_2),\\
b_1 \otimes \tilde{e}_i b_2 &{\rm if}\ \varphi_i (b_1) < \varepsilon_i (b_2),
\end{cases}\\
&\tilde{f}_i(b_1 \otimes b_2) \coloneqq
\begin{cases}
\tilde{f}_i b_1 \otimes b_2 &{\rm if}\ \varphi_i (b_1) > \varepsilon_i (b_2),\\
b_1 \otimes \tilde{f}_i b_2 &{\rm if}\ \varphi_i (b_1) \le \varepsilon_i (b_2);
\end{cases}
\end{align*}
here, $b_1 \otimes b_2$ stands for an ordered pair $(b_1, b_2)$, and we set $b_1 \otimes 0 = 0 \otimes b_2 = 0$.

Let $P_+ \subset P$ be the set of dominant integral weights, $B^- \subset G$ the Borel subgroup opposite to $B$, and $e_i, f_i, h_i \in \mathfrak{g}$, $i \in I$, the Chevalley generators such that $\{e_i, h_i \mid i \in I\} \subset {\rm Lie}(B)$ and $\{f_i, h_i \mid i \in I\} \subset {\rm Lie}(B^-)$. For $\lambda \in P_+$, we denote by $V(\lambda)$ the irreducible highest weight $G$-module over $\c$ with highest weight $\lambda$ and with highest weight vector $v_{\lambda}$. Lusztig \cite{Lus_can, Lus_quivers, Lus1} and Kashiwara \cite{Kas1,Kas2,Kas3} constructed a specific $\c$-basis of $V(\lambda)$ via the quantized enveloping algebra associated with $\mathfrak{g}$. This is called (the specialization at $q = 1$ of) the \emph{lower global basis} ($=$ the \emph{canonical basis}), and denoted by $\{G_{\lambda}^{\rm low}(b) \mid b \in \mathcal{B}(\lambda)\} \subset V(\lambda)$. The index set $\mathcal{B}(\lambda)$ has a crystal structure, which satisfies the following conditions:
\begin{align*}
&{\rm wt}(b_\lambda) = \lambda,\\
&\varepsilon_i(b) = \max\{k \in \mathbb{Z}_{\ge 0} \mid \tilde{e}_i ^k b \neq 0\},\\ 
&\varphi_i (b) = \max\{k \in \mathbb{Z}_{\ge 0} \mid \tilde{f}_i ^k b \neq 0\},\\
&e_i \cdot G^{\rm low} _{\lambda} (b) \in \c^\times G^{\rm low} _{\lambda} (\tilde{e}_i b) + \sum_{\substack{b^\prime \in \mathcal{B}(\lambda);\ {\rm wt}(b^\prime) = {\rm wt}(b) + \alpha_i,\\ \varphi_i (b^\prime) > \varphi_i (b) + 1}} \c G^{\rm low} _{\lambda} (b^\prime),\\
&f_i \cdot G^{\rm low} _{\lambda} (b) \in \c^\times G^{\rm low} _{\lambda} (\tilde{f}_i b) + \sum_{\substack{b^\prime \in \mathcal{B}(\lambda);\ {\rm wt}(b^\prime) = {\rm wt}(b) -\alpha_i,\\ \varepsilon_i (b^\prime) > \varepsilon_i (b) + 1}} \c G^{\rm low} _{\lambda} (b^\prime)
\end{align*}
for $i \in I$ and $b \in \mathcal{B}(\lambda)$, where $\c^\times \coloneqq \c \setminus \{0\}$, $G^{\rm low} _{\lambda} (0) \coloneqq 0$ if $\tilde{e}_i b = 0$ or $\tilde{f}_i b = 0$, and $b_\lambda \in \mathcal{B}(\lambda)$ is given by $G^{\rm low} _{\lambda} (b_\lambda) \in \c^\times v_\lambda$. We call $\mathcal{B}(\lambda)$ the \emph{crystal basis} for $V(\lambda)$; see \cite{Kas5} for a survey on lower global bases and crystal bases. 

Fix a reduced word ${\bf i} = (i_1, \ldots, i_N) \in I^N$ for the longest element $w_0 \in W$, and consider a sequence ${\bf j} = (\ldots, j_k, \ldots, j_{N+1}, j_N, \ldots, j_1)$ of elements in $I$ such that $j_k = i_{N-k+1}$ for $1 \le k \le N$, $j_k \neq j_{k+1}$ for all $k \ge 1$, and the cardinality of $\{k \ge 1 \mid j_k = i\}$ is $\infty$ for every $i \in I$. Following \cite{Kas4} and \cite{NZ}, we associate to ${\bf j}$ a crystal structure on \[\z^{\infty} \coloneqq \{(\ldots, a_k, \ldots, a_2, a_1) \mid a_k \in \z\ {\rm for}\ k \ge 1\ {\rm and}\ a_k = 0\ {\rm for}\ k \gg 0\}\] as follows. For $k \ge 1$, $i \in I$, and ${\bf a} = (\ldots, a_l, \ldots, a_2, a_1) \in \z^\infty$, we set 
\begin{align*}
&\sigma_k({\bf a}) \coloneqq a_k + \sum_{l > k} c_{j_k, j_l} a_l \in \z,\\
&\sigma^{(i)}({\bf a}) \coloneqq \max\{\sigma_k({\bf a}) \mid k \ge 1,\ j_k = i\} \in \z,\ {\rm and}\\
&M^{(i)}({\bf a}) \coloneqq \{k \ge 1 \mid j_k = i,\ \sigma_k({\bf a}) = \sigma^{(i)}({\bf a})\}.
\end{align*}
Since $a_l = 0$ for $l \gg 0$, the integers $\sigma_k({\bf a}), \sigma^{(i)}({\bf a})$ are well-defined; also, we have $\sigma^{(i)}({\bf a}) \ge 0$. Moreover, $M^{(i)}({\bf a})$ is a finite set if and only if $\sigma^{(i)}({\bf a}) > 0$. Define a crystal structure on $\z^\infty$ by 
\begin{align*}
&{\rm wt}({\bf a}) \coloneqq - \sum_{k = 1} ^\infty a_k \alpha_{j_k},\ \varepsilon_i({\bf a}) \coloneqq \sigma^{(i)}({\bf a}),\ \varphi_i({\bf a}) \coloneqq \varepsilon_i ({\bf a}) + \langle {\rm wt}({\bf a}), h_i \rangle,\ {\rm and}\\
&\tilde{e}_i {\bf a} \coloneqq 
\begin{cases}
(a_k - \delta_{k, \max M^{(i)}({\bf a})})_{k \ge 1} &{\rm if}\ \sigma^{(i)}({\bf a}) > 0,\\
0 &{\rm otherwise},
\end{cases}\\
&\tilde{f}_i {\bf a} \coloneqq (a_k + \delta_{k, \min M^{(i)}({\bf a})})_{k \ge 1} 
\end{align*}
for $i \in I$ and ${\bf a} = (\ldots, a_k, \ldots, a_2, a_1) \in \z^\infty$, where $\delta_{k, l}$ is the Kronecker delta; we denote this crystal by $\z^\infty _{{\bf j}}$. For $k \ge 1$, we set ${\bf j}_{\ge k} \coloneqq (\ldots, j_l, \ldots, j_{k+1}, j_k)$. Then, we see that the crystal $\z^\infty _{{\bf j}}$ is naturally isomorphic to the tensor product $\z^\infty _{{\bf j}_{\ge k}} \otimes \widetilde{\mathcal{B}}_{j_{k-1}} \otimes \cdots \otimes \widetilde{\mathcal{B}}_{j_1}$ for all $k \ge 2$. 

\vspace{2mm}\begin{prop}[{see \cite[Theorem 3.2]{Nak1} and \cite[Proposition 3.1]{Nak2}}]\label{p:polyhedral realizations}
For $\lambda \in P_+$, the following hold.
\begin{enumerate}
\item[{\rm (1)}] There exists a unique strict embedding of crystals \[\widetilde{\Psi}_{\bf j} \colon \mathcal{B} (\lambda) \lhook\joinrel\rightarrow \z^\infty _{\bf j} \otimes R_\lambda\] such that $\widetilde{\Psi}_{\bf j} (b_\lambda) = (\ldots, 0, \ldots, 0, 0) \otimes r_\lambda$.
\item[{\rm (2)}] If $(\ldots, a_k, \ldots, a_2, a_1) \otimes r_\lambda \in \widetilde{\Psi}_{\bf j} (\mathcal{B} (\lambda))$, then $a_k = 0$ for all $k > N$.
\end{enumerate} 
\end{prop}\vspace{2mm}

The embedding $\widetilde{\Psi}_{\bf j}$ (resp., the image $\widetilde{\Psi}_{\bf j}(\mathcal{B}(\lambda))$) is called the \emph{Kashiwara embedding} (resp., the \emph{polyhedral realization}) of $\mathcal{B}(\lambda)$ with respect to ${\bf j}$. 

\vspace{2mm}\begin{rem}\normalfont
We may regard Proposition \ref{p:polyhedral realizations} (1) as a definition of the crystal $\mathcal{B}(\lambda)$, that is, $\mathcal{B}(\lambda)$ is identified with \[\{\tilde{f}_{k_1} \cdots \tilde{f}_{k_l} ((\ldots, 0, 0) \otimes r_\lambda) \mid l \ge 0,\ k_1, \ldots, k_l \in I\} \setminus \{0\} \subset \z^\infty _{\bf j} \otimes R_\lambda\] as a set, and its crystal structure is given by that on $\z^\infty _{\bf j} \otimes R_\lambda$.
\end{rem}

\vspace{2mm}\begin{defi}\normalfont\label{definition2}
We define $\Psi_{\bf i} \colon \mathcal{B}(\lambda) \hookrightarrow \z^N$, $b \mapsto (a_1, a_2, \ldots, a_N)$, by \[\widetilde{\Psi}_{\bf j} (b) = (\ldots, 0, 0, a_1, a_2, \ldots, a_N) \otimes r_\lambda;\] this is also called the \emph{Kashiwara embedding} of $\mathcal{B}(\lambda)$ with respect to ${\bf i}$.
\end{defi}\vspace{2mm}

Note that the embedding $\Psi_{\bf i}$ is independent of the choice of an extension ${\bf j}$ by \cite[Sect.\ 2.4]{NZ}.

\vspace{2mm}\begin{defi}[{see \cite[Definition 2.15]{FN}}]\normalfont\label{definition3}
Let ${\bf i} \in I^N$ be a reduced word for $w_0$, and $\lambda \in P_+$. Define a subset $\mathcal{S}_{\bf i} (\lambda) \subset \z_{>0} \times \z^N$ by \[\mathcal{S}_{\bf i} (\lambda) \coloneqq \bigcup_{k >0} \{(k, \Psi_{\bf i}(b)) \mid b \in \mathcal{B}(k\lambda)\},\] and denote by $\mathcal{C}_{\bf i} (\lambda) \subset \r_{\ge 0} \times \r^N$ the smallest real closed cone containing $\mathcal{S}_{\bf i} (\lambda)$. Now let us define a subset $\Delta_{\bf i} (\lambda) \subset \r^N$ by \[\Delta_{\bf i} (\lambda) \coloneqq \{{\bf a} \in \r^N \mid (1, {\bf a}) \in \mathcal{C}_{\bf i} (\lambda)\}.\] The set $\Delta_{\bf i} (\lambda)$ is called the \emph{Nakashima-Zelevinsky polytope} associated with ${\bf i}$ and $\lambda$. 
\end{defi}\vspace{2mm}

\begin{prop}[{\cite[Corollaries 2.18 (2), 2.20, and 4.3]{FN}}]\label{p:convexity}
Let ${\bf i} \in I^N$ be a reduced word for $w_0$, and $\lambda \in P_+$. 
\begin{enumerate}
\item[{\rm (1)}] The real closed cone $\mathcal{C}_{\bf i} (\lambda)$ is a rational convex polyhedral cone, and the equality $\mathcal{S}_{\bf i} (\lambda) = \mathcal{C}_{\bf i} (\lambda) \cap (\z_{>0} \times \z^N)$ holds.
\item[{\rm (2)}] The Nakashima-Zelevinsky polytope $\Delta_{\bf i} (\lambda)$ is a rational convex polytope, and the equality $\Delta_{\bf i} (\lambda) \cap \z^N = \Psi_{\bf i} (\mathcal{B}(\lambda))$ holds.
\end{enumerate}
\end{prop}\vspace{2mm}

\begin{rem}\normalfont
In the case that $({\bf j}, \lambda)$ is ample (see \cite[Sect.\ 4.2]{Nak1} for the definition), a system of explicit linear inequalities defining $\Delta_{\bf i} (\lambda)$ is given by \cite[Theorem 4.1]{Nak1} (see also \cite[Corollary 5.3]{FN}). Note that in order to prove Proposition \ref{p:convexity}, the ampleness of $({\bf j}, \lambda)$ is not necessary.
\end{rem}\vspace{2mm}

\begin{ex}[{\cite{Nak1}}]\normalfont\label{e:GZ}
Let $G = SL_{n+1}(\c)$, and $\lambda \in P_+$. We consider a specific reduced word ${\bf i} = (1, 2, 1, 3, 2, 1, \ldots, n, n-1, \ldots, 1)$ for $w_0$. Then, by \cite[Theorem 6.1]{Nak1} (see also \cite[Corollary 2.7]{Nak3}), the Nakashima-Zelevinsky polytope $\Delta_{\bf i} (\lambda)$ is identical to the set of $(a_n ^{(1)}, a_{n -1} ^{(2)}, a_{n -1} ^{(1)}, \ldots, a_1 ^{(n)}, \ldots, a_1 ^{(1)}) \in \r^N$ satisfying the following conditions:
\[\begin{matrix}
\lambda_{\ge 1} & & \lambda_{\ge 2} & & \cdots & & & \lambda_{\ge n} & & 0\\
 & a_1 ^{(1)} + \lambda_{\ge 2} & & a_2 ^{(1)} + \lambda_{\ge 3} & & \cdots & & & a_n ^{(1)} & \\
 & & a_1 ^{(2)} + \lambda_{\ge 3} & & \cdots & & & a_{n -1} ^{(2)} & & \\
 & & & \ddots & & \ldots & & & & \\
 & & & & a_1 ^{(n-1)} + \lambda_{\ge n} & & a_2 ^{(n-1)} & & & \\
 & & & & & a_1 ^{(n)}, & & & & 
\end{matrix}\]
where $N \coloneqq \frac{n(n +1)}{2}$, $\lambda_{\ge k} \coloneqq \sum_{k \le l \le n} \langle \lambda, h_l \rangle$ for $1 \le k \le n$, and the notation 
\[\begin{matrix}
a & & c\\
 & b & 
\end{matrix}\]
means that $a \ge b \ge c$. This implies that the translation \[\Delta_{\bf i} (\lambda) + (0, \underbrace{0, \lambda_{\ge n}}_{2}, \underbrace{0, \lambda_{\ge n}, \lambda_{\ge n -1}}_{3}, \ldots, \underbrace{0, \lambda_{\ge n}, \lambda_{\ge n -1}, \ldots, \lambda_{\ge 2}}_{n})\] of the Nakashima-Zelevinsky polytope is identical to the Gelfand-Zetlin polytope $GZ(\overline{\lambda})$ associated with the non-increasing sequence $\overline{\lambda} \coloneqq (\lambda_{\ge 1}, \lambda_{\ge 2}, \ldots, \lambda_{\ge n}, 0)$.
\end{ex}\vspace{2mm}

For $w \in W$ and $\lambda \in P_+$, let $v_{w\lambda} \in V(\lambda)$ be a weight vector of weight $w \lambda$, which is called an \emph{extremal weight vector}. We define a $B$-submodule $V_w(\lambda) \subset V(\lambda)$ (resp., a $B^-$-submodule $V^w(\lambda) \subset V(\lambda)$) by 
\begin{align*}
&V_w(\lambda) \coloneqq \sum_{b \in B} \c b v_{w\lambda}\\
({\rm resp.},\ &V^w(\lambda) \coloneqq \sum_{b \in B^-} \c b v_{w\lambda});
\end{align*}
this is called the \emph{Demazure module} (resp., the \emph{opposite Demazure module}) associated with $w \in W$. By \cite[Proposition 3.2.3 (i) and equation (4.1)]{Kas4}, there uniquely exists a subset $\mathcal{B}_w(\lambda)$ (resp., $\mathcal{B}^w(\lambda)$) of $\mathcal{B}(\lambda)$ such that 
\begin{align*}
&V_w(\lambda) = \sum_{b \in \mathcal{B}_w(\lambda)} \c G^{\rm low} _\lambda(b)\\
({\rm resp.},\ &V^w(\lambda) = \sum_{b \in \mathcal{B}^w(\lambda)} \c G^{\rm low} _\lambda(b));
\end{align*}
this subset $\mathcal{B}_w(\lambda)$ (resp., $\mathcal{B}^w(\lambda)$) is called a \emph{Demazure crystal} (resp., an \emph{opposite Demazure crystal}). Let $b_{w \lambda} \in \mathcal{B}(\lambda)$ denote the extremal weight element of weight $w \lambda$, that is, $b_{w \lambda}$ is a unique element in $\mathcal{B}(\lambda)$ such that $G^{\rm low} _{\lambda} (b_{w \lambda}) \in \c^\times v_{w \lambda}$. Then, we have \[\mathcal{B}_w(\lambda) \cap \mathcal{B}^w(\lambda) = \{b_{w \lambda}\}.\] Let $\{s_i \mid i \in I\} \subset W$ be the set of simple reflections. The following is a collection of fundamental properties of Demazure crystals and opposite Demazure crystals.

\vspace{2mm}\begin{prop}[{\cite[Propositions 3.2.3 (ii), (iii) and 4.2]{Kas4}}]\label{p:properties of Demazure}
Let $w \in W$, and $\lambda \in P_+$.
\begin{enumerate}
\item[{\rm (1)}] $\tilde{e}_i \mathcal{B}_w(\lambda) \subset \mathcal{B}_w(\lambda) \cup \{0\}$ and $\tilde{f}_i \mathcal{B}^w(\lambda) \subset \mathcal{B}^w(\lambda) \cup \{0\}$ for all $i \in I$.  
\item[{\rm (2)}] If $s_i w < w$, then 
\begin{align*}
&\mathcal{B}_w(\lambda) = \bigcup_{k \ge 0} \tilde{f}_i ^k \mathcal{B}_{s_i w}(\lambda) \setminus \{0\},\\
&\mathcal{B}^{s_i w}(\lambda) = \bigcup_{k \ge 0} \tilde{e}_i ^k \mathcal{B}^w(\lambda) \setminus \{0\}.
\end{align*}
\item[{\rm (3)}] Let ${\bf i} = (i_1, \ldots, i_r) \in I^r$ be a reduced word for $w \in W$. Then, \[\mathcal{B}_w(\lambda) = \{\tilde{f}_{i_1} ^{a_1} \cdots \tilde{f}_{i_r} ^{a_r} b_\lambda \mid a_1, \ldots, a_r \in \mathbb{Z}_{\ge 0}\} \setminus \{0\}.\]
\item[{\rm (4)}] Let ${\bf i} = (i_1, \ldots, i_r) \in I^r$ be a reduced word for $w w_0 \in W$. Then, \[\mathcal{B}^w(\lambda) = \{\tilde{e}_{i_1} ^{a_1} \cdots \tilde{e}_{i_r} ^{a_r} b_{w_0 \lambda} \mid a_1, \ldots, a_r \in \mathbb{Z}_{\ge 0}\} \setminus \{0\}.\] 
\end{enumerate}
\end{prop}\vspace{2mm}

For $\lambda \in P_+$, the crystal $\mathcal{B}(-w_0 \lambda)$ is identified with the dual crystal of $\mathcal{B}(\lambda)$ (see \cite[Sect.\ 1.2]{Kas4} for more details). Under this identification, the opposite Demazure crystals of $\mathcal{B}(\lambda)$ correspond to the Demazure crystals of $\mathcal{B}(-w_0 \lambda)$. For $i \in I$, a subset $S \subset \mathcal{B}(\lambda)$ is called an \emph{$i$-string} if there exists $b_S ^{\rm high} \in S$ such that $\tilde{e}_i b_S ^{\rm high} = 0$, and such that \[S = \{\tilde{f}_i ^k b_S ^{\rm high} \mid k \in \z_{\ge 0}\} \setminus \{0\}.\] This element $b_S ^{\rm high}$ is called the \emph{highest weight element} of $S$; similarly, the \emph{lowest weight element} $b_S ^{\rm low} \in S$ is defined by $\tilde{f}_i b_S ^{\rm low} = 0$. The following is called the \emph{string property} of Demazure crystals and opposite Demazure crystals.

\vspace{2mm}\begin{prop}[{see \cite[Proposition 3.3.5]{Kas4}}]\label{p:string property}
Let $w \in W$, $\lambda \in P_+$, and $i \in I$.
\begin{enumerate}
\item[{\rm (1)}] For each $i$-string $S$ of $\mathcal{B}(\lambda)$ with highest weight element $b_S ^{\rm high}$, the intersection $\mathcal{B}_w(\lambda) \cap S$ is either $\emptyset$, $S$, or $\{b_S ^{\rm high}\}$.
\item[{\rm (2)}] For each $i$-string $S$ of $\mathcal{B}(\lambda)$ with lowest weight element $b_S ^{\rm low}$, the intersection $\mathcal{B}^w(\lambda) \cap S$ is either $\emptyset$, $S$, or $\{b_S ^{\rm low}\}$.
\end{enumerate}
\end{prop}\vspace{2mm}

Let ${\bf i} = (i_1, \ldots, i_N) \in I^N$ be a reduced word for $w_0$, and $\lambda \in P_+$. We write $w_{\ge k} \coloneqq s_{i_k} \cdots s_{i_N} \in W$ and $x_k \coloneqq -\langle w_{\ge k} \lambda, h_{i_k} \rangle$ for $1 \le k \le N$.

\vspace{2mm}\begin{thm}[{\cite[Theorem 4.1]{Nak2}}]\label{t:Nakashima}
Let ${\bf i} = (i_1, \ldots, i_N) \in I^N$ be a reduced word for $w_0$, $\lambda \in P_+$, and $1 \le k \le N$. Then, the image $\Psi_{\bf i} (b_{w_{\ge k} \lambda})$ is given by \[\Psi_{\bf i} (b_{w_{\ge k} \lambda}) = (0, \ldots, 0, x_k, \ldots, x_N).\]
\end{thm}\vspace{2mm}

For $1 \le k \le N$, we define
\begin{align*}
&\pi_{\ge k} \colon \mathcal{B}(\lambda) \rightarrow \widetilde{\mathcal{B}}_{i_k} \otimes \cdots \otimes \widetilde{\mathcal{B}}_{i_N} \otimes R_\lambda\ {\rm and}\\
&\pi_{\le k} \colon \mathcal{B}(\lambda) \rightarrow \z_{{\bf j}_{\ge N+1}} ^\infty \otimes \widetilde{\mathcal{B}}_{i_1} \otimes \cdots \otimes \widetilde{\mathcal{B}}_{i_k}
\end{align*}
by $\pi_{\ge k}(b) \coloneqq b_2$ and $\pi_{\le k}(b^\prime) \coloneqq b_1 ^\prime$ for $b, b^\prime \in \mathcal{B}(\lambda)$ such that 
\begin{align*}
&\widetilde{\Psi}_{\bf j} (b) = b_1 \otimes b_2 \in (\z_{{\bf j}_{\ge N+1}} ^\infty \otimes \widetilde{\mathcal{B}}_{i_1} \otimes \cdots \otimes \widetilde{\mathcal{B}}_{i_{k-1}}) \otimes (\widetilde{\mathcal{B}}_{i_k} \otimes \cdots \otimes \widetilde{\mathcal{B}}_{i_N} \otimes R_\lambda)\ {\rm and}\\ 
&\widetilde{\Psi}_{\bf j} (b^\prime) = b_1 ^\prime \otimes b_2 ^\prime \in (\z_{{\bf j}_{\ge N+1}} ^\infty \otimes \widetilde{\mathcal{B}}_{i_1} \otimes \cdots \otimes \widetilde{\mathcal{B}}_{i_k}) \otimes (\widetilde{\mathcal{B}}_{i_{k+1}} \otimes \cdots \otimes \widetilde{\mathcal{B}}_{i_N} \otimes R_\lambda),
\end{align*}
respectively. In addition, we set $\pi_{\ge 0} = \pi_{\le N+1} = \widetilde{\Psi}_{\bf j}$, and 
\begin{align*}
&\pi_{\ge N+1} \colon \mathcal{B}(\lambda) \rightarrow R_\lambda,\ b \mapsto r_\lambda,\\
&\pi_{\le 0} \colon \mathcal{B}(\lambda) \rightarrow \z_{{\bf j}_{\ge N+1}} ^\infty,\ b \mapsto (\ldots, 0, \ldots, 0, 0).
\end{align*} 
We write ${\bf x}_{\ge k} \coloneqq \pi_{\ge k}(b_{w_{\ge k} \lambda})$ for $1 \le k \le N$. 

\vspace{2mm}\begin{lem}\label{l:Nakashima}
The following equalities hold for $2 \le k \le N$$:$ 
\begin{align*}
&\varepsilon_{i_{k-1}}({\bf x}_{\ge k-1}) = x_{k-1},\ \tilde{e}_{i_{k-1}} ^{x_{k-1}} {\bf x}_{\ge k-1} = (0)_{i_{k-1}} \otimes {\bf x}_{\ge k},\\
&\varphi_{i_{k-1}}((0)_{i_{k-1}} \otimes {\bf x}_{\ge k}) = x_{k-1},\ \tilde{f}_{i_{k-1}} ^{x_{k-1}} ((0)_{i_{k-1}} \otimes {\bf x}_{\ge k}) = {\bf x}_{\ge k-1}.
\end{align*}
\end{lem}

\begin{proof}
Since $\varepsilon_{i_{k-1}}(b_{w_{\ge k-1} \lambda}) = x_{k-1} = \varepsilon_{i_{k-1}} ((x_{k-1})_{i_{k-1}})$, and 
\begin{align*}
\widetilde{\Psi}_{\bf j} (b_{w_{\ge k} \lambda}) &= \tilde{e}_{i_{k-1}} ^{\varepsilon_{i_{k-1}}(b_{w_{\ge k-1} \lambda})} \widetilde{\Psi}_{\bf j} (b_{w_{\ge k-1} \lambda})\\
&= \tilde{e}_{i_{k-1}} ^{x_{k-1}} \widetilde{\Psi}_{\bf j} (b_{w_{\ge k-1} \lambda}),
\end{align*}
Theorem \ref{t:Nakashima} and the tensor product rule for crystals imply that $\varepsilon_{i_{k-1}}({\bf x}_{\ge k-1}) = x_{k-1},\ \tilde{e}_{i_{k-1}} ^{x_{k-1}} {\bf x}_{\ge k-1} = (0)_{i_{k-1}} \otimes {\bf x}_{\ge k}$. The other assertions of the lemma follow from these and $\varphi_{i_{k-1}}(b_{w_{\ge k} \lambda}) = x_{k-1}$.
\end{proof}

\begin{prop}\label{p:polytopes for opposite}
The following equality holds for $1 \le k \le N$$:$ 
\[\Psi_{\bf i} (\mathcal{B}^{w_{\ge k}} (\lambda)) = \{{\bf a} = (a_1, \ldots, a_N) \in \Psi_{\bf i} (\mathcal{B}(\lambda)) \mid a_l = x_l\ {\it for\ all}\ k \le l \le N\}.\]
\end{prop}

\begin{proof}
We will prove that \[\mathcal{B}^{w_{\ge k}} (\lambda) = \{b \in \mathcal{B}(\lambda) \mid \pi_{\ge k}(b) = {\bf x}_{\ge k}\}\] for $1 \le k \le N$. We proceed by induction on $k$. If $k = 1$, then the assertion is obvious since $\mathcal{B}^{w_{\ge 1}} (\lambda) = \mathcal{B}^{w_0} (\lambda) = \{b_{w_0 \lambda}\}$ and $\Psi_{\bf i} (b_{w_0 \lambda}) = (x_1, \ldots, x_N)$ by Theorem \ref{t:Nakashima}. We assume that $k > 1$, and that \[\mathcal{B}^{w_{\ge k-1}} (\lambda) = \{b \in \mathcal{B}(\lambda) \mid \pi_{\ge k-1}(b) = {\bf x}_{\ge k-1}\}.\] Take $b \in \mathcal{B}^{w_{\ge k}} (\lambda)$. Then, we see by Proposition \ref{p:properties of Demazure} that $\tilde{f}_{i_{k-1}} ^{\varphi_{i_{k-1}}(b)} b \in \mathcal{B}^{w_{\ge k-1}} (\lambda)$; hence the equality $\pi_{\ge k-1}(\tilde{f}_{i_{k-1}} ^{\varphi_{i_{k-1}}(b)} b) = {\bf x}_{\ge k-1}$ holds. From this and Lemma \ref{l:Nakashima}, we deduce that \[\pi_{\ge k}(b) = \pi_{\ge k}(\tilde{e}_{i_{k-1}} ^{\varphi_{i_{k-1}}(b)} \tilde{f}_{i_{k-1}} ^{\varphi_{i_{k-1}}(b)}b) = {\bf x}_{\ge k}.\] Conversely, take $b \in \mathcal{B}(\lambda)$ such that $\pi_{\ge k}(b) = {\bf x}_{\ge k}$. Then, we have $\widetilde{\Psi}_{\bf j}(b) = \pi_{\le k-2}(b) \otimes (a)_{i_{k-1}} \otimes {\bf x}_{\ge k}$ for some $0 \le a \le x_{k-1}$. By Lemma \ref{l:Nakashima}, it follows that $\varphi_{i_{k-1}}((a)_{i_{k-1}} \otimes {\bf x}_{\ge k}) = x_{k-1} -a$, and that $\tilde{f}_{i_{k-1}} ^{x_{k-1} -a} ((a)_{i_{k-1}} \otimes {\bf x}_{\ge k}) = {\bf x}_{\ge k-1}$. Hence by the tensor product rule for crystals, we deduce that 
\begin{align*}
\tilde{f}_{i_{k-1}} ^{\varphi_{i_{k-1}}(b)} \widetilde{\Psi}_{\bf j}(b) &= \tilde{f}_{i_{k-1}} ^{\varphi_{i_{k-1}}(b) -(x_{k-1} -a)} \pi_{\le k-2}(b) \otimes {\bf x}_{\ge k-1}.
\end{align*}
From this, it follows that $\pi_{\ge k-1}(\tilde{f}_{i_{k-1}} ^{\varphi_{i_{k-1}}(b)} b) = {\bf x}_{\ge k-1}$, and hence that $\tilde{f}_{i_{k-1}} ^{\varphi_{i_{k-1}}(b)} b \in \mathcal{B}^{w_{\ge k-1}} (\lambda)$. By Proposition \ref{p:properties of Demazure}, this implies that $b \in \mathcal{B}^{w_{\ge k}} (\lambda)$. These prove the proposition.
\end{proof}

\begin{cor}\label{c:additivity of opposite}
For all $\lambda, \mu \in P_+$ and $1 \le k \le N$, the following holds$:$
\[\Psi_{\bf i} (\mathcal{B}^{w_{\ge k}} (\lambda)) + \Psi_{\bf i} (\mathcal{B}^{w_{\ge k}} (\mu)) \subset \Psi_{\bf i} (\mathcal{B}^{w_{\ge k}} (\lambda + \mu)).\]
\end{cor}

\begin{proof}
Since $-\langle w_{\ge l} (\lambda +\mu), h_{i_l} \rangle = -\langle w_{\ge l} \lambda, h_{i_l} \rangle -\langle w_{\ge l} \mu, h_{i_l} \rangle$ for $k \le l \le N$, Proposition \ref{p:polytopes for opposite} implies that it suffices to show that $\Psi_{\bf i} (\mathcal{B}^{w_{\ge k}} (\lambda)) + \Psi_{\bf i} (\mathcal{B}^{w_{\ge k}} (\mu)) \subset \Psi_{\bf i} (\mathcal{B}(\lambda + \mu))$. However, this follows immediately by the additivity of $\Psi_{\bf i}$ (see \cite[Theorem 4.1]{FN}).
\end{proof}

\section{Main result}

\subsection{Statement of the main result}

Let ${\bf i} = (i_1, \ldots, i_N) \in I^N$ be a reduced word for $w_0$. For $i \in I$, recall that $d_i$ is the number of $1 \le k \le N$ such that $i_k = i$. For $\lambda \in P_+$, we write $\lambda = \sum_{i \in I} \hat{\lambda}_i d_i \alpha_i$, and set
\begin{align*}
&{\bf x}_\lambda \coloneqq \Psi_{\bf i} (b_{w_0 \lambda}) = (x_1, \ldots, x_N),\\
&{\bf a}_\lambda \coloneqq -{\bf x}_\lambda + (\hat{\lambda}_{i_1}, \ldots, \hat{\lambda}_{i_N}).
\end{align*}
The following is the main result of this paper.

\vspace{2mm}\begin{thm}\label{t:main result}
Let ${\bf i} = (i_1, \ldots, i_N) \in I^N$ be a reduced word for $w_0$, and $\lambda \in P_+$. Assume that the Nakashima-Zelevinsky polytope $\Delta_{\bf i} (\lambda)$ is a parapolytope.
\begin{enumerate}
\item[{\rm (1)}] The polytope $\Delta_{\bf i} (\lambda)$ is a lattice polytope.
\item[{\rm (2)}] The polytope $D_{i_N} ^{(N)} \cdots D_{i_1} ^{(1)} ({\bf a}_\lambda)$ is well-defined.
\item[{\rm (3)}] The following equality holds$:$ \[D_{i_N} ^{(N)} \cdots D_{i_1} ^{(1)} ({\bf a}_\lambda) = -\Delta_{\bf i} (\lambda) + (\hat{\lambda}_{i_1}, \ldots, \hat{\lambda}_{i_N}).\] 
\end{enumerate}
\end{thm}\vspace{2mm}

We prove Theorem \ref{t:main result} in the next subsection. In the rest of this subsection, we give some examples of $\Delta_{\bf i} (\lambda)$ which are parapolytopes.

\vspace{2mm}\begin{ex}\normalfont\label{e:main type A}
Let $G = SL_{n+1}(\c)$, and ${\bf i} = (1, 2, 1, 3, 2, 1, \ldots, n, n-1, \ldots, 1)$. Then, the Nakashima-Zelevinsky polytope $\Delta_{\bf i} (\lambda)$ is a parapolytope for all $\lambda \in P_+$ by Example \ref{e:GZ}. 
\end{ex}\vspace{2mm}

\begin{ex}[{\cite{Hos}}]\normalfont\label{e:main type BCD}
Let $G$ be of type $B_n$, $C_n$, or $D_n$. We identify the set of vertices of the Dynkin diagram with $\{1, \ldots, n\}$ as follows:
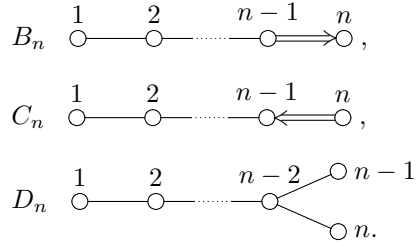
\begin{align*}
&B_n\ \begin{xy}
\ar@{-} (50,0) *++!D{1} *\cir<3pt>{};
(60,0) *++!D{2} *\cir<3pt>{}="C"
\ar@{-} "C";(65,0) \ar@{.} (65,0);(70,0)^*!U{}
\ar@{-} (70,0);(75,0) *++!D{n-1} *\cir<3pt>{}="D"
\ar@{=>} "D";(85,0) *++!D{n} *\cir<3pt>{}="E"
\end{xy}\hspace{-1mm},\\ 
&C_n\ \begin{xy}
\ar@{-} (50,0) *++!D{1} *\cir<3pt>{};
(60,0) *++!D{2} *\cir<3pt>{}="C"
\ar@{-} "C";(65,0) \ar@{.} (65,0);(70,0)^*!U{}
\ar@{-} (70,0);(75,0) *++!D{n-1} *\cir<3pt>{}="D"
\ar@{<=} "D";(85,0) *++!D{n} *\cir<3pt>{}="E"
\end{xy}\hspace{-1mm},\\ 
&D_n\ \begin{xy}
\ar@{-} (20,0) *++!D{1} *\cir<3pt>{};
(30,0) *++!D{2} *\cir<3pt>{}="B"
\ar@{-} "B";(35,0) \ar@{.} (35,0);(40,0)^*!U{}
\ar@{-} (40,0);(45,0) *++!D{n-2} *\cir<3pt>{}="C"
\ar@{-} "C";(54,4) *++!L{n-1} *\cir<3pt>{}
\ar@{-} "C";(54,-4) *++!L{n.} *\cir<3pt>{},
\end{xy}
\end{align*}
We take a reduced word ${\bf i}$ for $w_0$ to be \[{\bf i} = (n, n-1, \ldots, 1, n, n-1, \ldots, 1, \ldots, n, n-1, \ldots, 1),\] where ${\bf i} \in I^{n^2}$ if $G$ is of type $B_n$ or $C_n$, and ${\bf i} \in I^{n (n-1)}$ if $G$ is of type $D_n$. 

If $G$ is of type $B_n$, then we see by \cite[Sect.\ I\hspace{-.1em}I\hspace{-.1em}I.A]{Hos} that $\Delta_{\bf i} (\lambda)$ is identical to the set of \[(a_n ^{(n)}, a_n ^{(n-1)}, \ldots, a_n ^{(1)}, \ldots, a_1 ^{(n)}, a_1 ^{(n-1)}, \ldots, a_1 ^{(1)}) \in \r^{n^2} _{\ge 0}\] satisfying the following inequalities:
\begin{align*}
&a_1 ^{(i)} \ge a_2 ^{(i-1)} \ge \cdots \ge a_i ^{(1)}\ {\rm for}\ 2 \le i \le n-1,\\
&a_j ^{(n)} \ge a_{j+1} ^{(n-1)} \ge \cdots \ge a_n ^{(j)}\ {\rm for}\ 1 \le j \le n-1,\\
&a_j ^{(n-j+1)} \ge a_j ^{(n-j+2)} \ge \cdots \ge a_j ^{(n)}\ {\rm for}\ 2 \le j \le n,\\
&\lambda_i \ge a_j ^{(i-j+1)} - a_j ^{(i-j)}\ {\rm for}\ 1 \le j \le i \le n-1,\\
&\lambda_n \ge a_{l} ^{(n)} - 2a_{l} ^{(n - 1)} + 2 \sum_{1 \le k \le l - 1} (a_{\mu_k + k -1} ^{(n - \mu_k + 1)} - a_{\mu_k + k -1} ^{(n - \mu_k)})\ {\rm for}\ l \ge 1,\ n \ge \mu_1 > \cdots > \mu_l = 1,\\
&\lambda_n \ge -a_{l} ^{(n)} + 2\sum_{1 \le k \le l} (a_{\mu_k + k -1} ^{(n - \mu_k + 1)} - a_{\mu_k + k -1} ^{(n - \mu_k)})\ {\rm for}\ l \ge 1,\ n \ge \mu_1 > \cdots > \mu_l > 1.
\end{align*}

If $G$ is of type $C_n$, then it follows by \cite[Sect.\ I\hspace{-.1em}I\hspace{-.1em}I.B]{Hos} that $\Delta_{\bf i} (\lambda)$ is identical to the set of \[(a_n ^{(n)}, a_n ^{(n-1)}, \ldots, a_n ^{(1)}, \ldots, a_1 ^{(n)}, a_1 ^{(n-1)}, \ldots, a_1 ^{(1)}) \in \r^{n^2} _{\ge 0}\] satisfying the following inequalities:
\begin{align*}
&a_1 ^{(i)} \ge a_2 ^{(i-1)} \ge \cdots \ge a_i ^{(1)}\ {\rm for}\ 2 \le i \le n-1,\\
&2a_j ^{(n)} \ge a_{j+1} ^{(n-1)} \ge \cdots \ge a_n ^{(j)}\ {\rm for}\ 1 \le j \le n-1,\\
&a_j ^{(n-j+1)} \ge a_j ^{(n-j+2)} \ge \cdots \ge a_j ^{(n-1)} \ge 2a_j ^{(n)}\ {\rm for}\ 2 \le j \le n,\\
&\lambda_i \ge a_j ^{(i-j+1)} - a_j ^{(i-j)}\ {\rm for}\ 1 \le j \le i \le n-1,\\
&\lambda_n \ge a_{l} ^{(n)} - a_{l} ^{(n - 1)} + \sum_{1 \le k \le l - 1} (a_{\mu_k + k -1} ^{(n - \mu_k + 1)} - a_{\mu_k + k -1} ^{(n - \mu_k)})\ {\rm for}\ l \ge 1,\ n \ge \mu_1 > \cdots > \mu_l = 1,\\
&\lambda_n \ge -a_{l} ^{(n)} + \sum_{1 \le k \le l} (a_{\mu_k + k -1} ^{(n - \mu_k + 1)} - a_{\mu_k + k -1} ^{(n - \mu_k)})\ {\rm for}\ l \ge 1,\ n \ge \mu_1 > \cdots > \mu_l > 1.
\end{align*}

If $G$ is of type $D_n$, then it follows by \cite[Sect.\ I\hspace{-.1em}I\hspace{-.1em}I.C]{Hos} that $\Delta_{\bf i} (\lambda)$ is identical to the set of \[(a_{n-1} ^{(n)}, a_{n-1} ^{(n-1)}, \ldots, a_{n-1} ^{(1)}, \ldots, a_1 ^{(n)}, a_1 ^{(n-1)}, \ldots, a_1 ^{(1)}) \in \r^{n (n-1)} _{\ge 0}\] satisfying the following inequalities:
\begin{align*}
&a_1 ^{(i)} \ge a_2 ^{(i-1)} \ge \cdots \ge a_i ^{(1)}\ {\rm for}\ 2 \le i \le n-2,\\
&a_j ^{(n-1)} + a_j ^{(n)} \ge a_{j+1} ^{(n-2)} \ge a_{j+2} ^{(n-3)} \ge \cdots \ge a_{n-1} ^{(j)}\ {\rm for}\ 1 \le j \le n-2,\\
&a_j ^{(n-j)} \ge a_j ^{(n-j+1)} \ge a_j ^{(n-j+2)} \ge \cdots \ge a_j ^{(n-2)} \ge a_j ^{(n-1)} + a_j ^{(n)}\ {\rm for}\ 2 \le j \le n-1,\\
&a_1 ^{(n-1)} \ge a_2 ^{(n)} \ge a_3 ^{(n-1)} \ge a_4 ^{(n)} \ge \cdots,\\
&a_1 ^{(n)} \ge a_2 ^{(n-1)} \ge a_3 ^{(n)} \ge a_4 ^{(n-1)} \ge \cdots,\\
&\lambda_i \ge a_j ^{(i-j+1)} - a_j ^{(i-j)}\ {\rm for}\ 1 \le j \le i \le n-2,\\
&\lambda_{n-1} \ge a_{1} ^{(n-1)} - a_{1} ^{(n - 2)},\\
&\lambda_{n} \ge a_{1} ^{(n)} - a_{1} ^{(n - 2)},\\
&\lambda_{n-1} \ge \max\{-a_{2l -1} ^{(n)}, a_{2l} ^{(n)} - a_{2l} ^{(n - 2)}\} + \sum_{1 \le k \le 2l -1} (a_{\mu_k + k -1} ^{(n - \mu_k)} - a_{\mu_k + k -1} ^{(n - \mu_k -1)}),\\ 
&\lambda_{n} \ge \max\{-a_{2l -1} ^{(n-1)}, a_{2l} ^{(n-1)} - a_{2l} ^{(n - 2)}\} + \sum_{1 \le k \le 2l -1} (a_{\mu_k + k -1} ^{(n - \mu_k)} - a_{\mu_k + k -1} ^{(n - \mu_k -1)})\\
&{\rm for}\ l \ge 1,\ n-1 \ge \mu_1 > \cdots > \mu_{2l -1} > 1,\\
&\lambda_{n-1} \ge \max\{-a_{2l} ^{(n-1)}, a_{2l +1} ^{(n-1)} - a_{2l +1} ^{(n - 2)}\} + \sum_{1 \le k \le 2l} (a_{\mu_k + k -1} ^{(n - \mu_k)} - a_{\mu_k + k -1} ^{(n - \mu_k -1)}),\\
&\lambda_{n} \ge \max\{-a_{2l} ^{(n)}, a_{2l +1} ^{(n)} - a_{2l +1} ^{(n - 2)}\} + \sum_{1 \le k \le 2l} (a_{\mu_k + k -1} ^{(n - \mu_k)} - a_{\mu_k + k -1} ^{(n - \mu_k -1)}),\\
&{\rm for}\ l \ge 1,\ n-1 \ge \mu_1 > \cdots > \mu_{2l} > 1.
\end{align*} 

In all cases, the Nakashima-Zelevinsky polytopes $\Delta_{\bf i} (\lambda)$, $\lambda \in P_+$, are parapolytopes.
\end{ex}\vspace{2mm}

\begin{ex}\normalfont\label{e:main type G}
Let $G$ be of type $G_2$. We set ${\bf i} \coloneqq (1, 2, 1, 2, 1, 2)$ and ${\bf i}^{\rm op} \coloneqq (2, 1, 2, 1, 2, 1)$. By \cite[Theorem 5.1]{Nak1}, the Nakashima-Zelevinsky polytopes $\Delta_{\bf i} (\lambda)$ and $\Delta_{{\bf i}^{\rm op}} (\lambda)$ are parapolytopes for all $\lambda \in P_+$.
\end{ex}\vspace{2mm}

\subsection{Proof of Theorem \ref{t:main result}}

We set $(\z^{d_i})^\perp \coloneqq (\r^{d_i})^\perp \cap \z^N$ for $i \in I$. Since $\Psi_{\bf i}(\mathcal{B}(\lambda)) = \Delta_{\bf i} (\lambda) \cap \z^N$ by Proposition \ref{p:convexity} (2), for $i \in I$ and ${\bf c} \in (\z^{d_i})^\perp$ such that $\Psi_{\bf i} (\mathcal{B}(\lambda)) \cap ({\bf c} + \z^{d_i}) \neq \emptyset$, there uniquely exist $\mu^{(i)} ({\bf c}) = (\mu_1 ^{(i)}({\bf c}), \ldots, \mu_{d_i} ^{(i)}({\bf c}))$, $\nu^{(i)} ({\bf c}) = (\nu_1 ^{(i)}({\bf c}), \ldots, \nu_{d_i} ^{(i)}({\bf c})) \in \z^{d_i}$ such that
\begin{align*}
&\Psi_{\bf i} (\mathcal{B}(\lambda)) \cap ({\bf c} + \z^{d_i}) = {\bf c} + \Pi_\z(\mu^{(i)} ({\bf c}), \nu^{(i)} ({\bf c})),
\end{align*}
where we write \[\Pi_\z(\mu^{(i)} ({\bf c}), \nu^{(i)} ({\bf c})) \coloneqq \{(a_1 ^{(i)}, \ldots, a^{(i)} _{d_i}) \in \z^{d_i} \mid \mu_l ^{(i)}({\bf c}) \le a_l ^{(i)} \le \nu_l ^{(i)}({\bf c}),\ 1 \le l \le d_i\}.\] Note that the subset $(\mathcal{B}(\lambda) \cap \Psi_{\bf i} ^{-1}({\bf c} + \z^{d_i})) \cup \{0\}$ of $\mathcal{B}(\lambda) \cup \{0\}$ is stable under $\tilde{e}_i$ and $\tilde{f}_i$ by the crystal structure on $\z^\infty _{\bf j} \otimes R_\lambda$. For $0 \le k \le N$, $i \in I$, and ${\bf c} \in (\z^{d_i})^\perp$ such that $\Psi_{\bf i} (\mathcal{B}^{w_{\ge k+1}} (\lambda)) \cap ({\bf c} + \z^{d_i}) \neq \emptyset$, Proposition \ref{p:polytopes for opposite} implies that there uniquely exist \[\mu^{(i, k)} ({\bf c}) = (\mu^{(i, k)} _1 ({\bf c}), \ldots, \mu^{(i, k)} _{d_i} ({\bf c})),\ \nu^{(i, k)} ({\bf c}) = (\nu^{(i, k)} _1 ({\bf c}), \ldots, \nu^{(i, k)} _{d_i} ({\bf c})) \in \z^{d_i}\] such that
\begin{align*}
&\Psi_{\bf i} (\mathcal{B}^{w_{\ge k+1}} (\lambda)) \cap ({\bf c} + \z^{d_i}) = {\bf c} + \Pi_\z(\mu^{(i, k)} ({\bf c}), \nu^{(i, k)} ({\bf c})),
\end{align*}
where we define $w_{\ge N+1} \in W$ to be the identity element. For $1 \le k \le N$ and ${\bf c} = (c_s)_{1 \le s \le N;\ i_s \neq i_k} \in (\z^{d_{i_k}})^\perp$ such that $\Psi_{\bf i} (\mathcal{B}^{w_{\ge k}} (\lambda)) \cap ({\bf c} + \z^{d_{i_k}}) \neq \emptyset$, we define $L_k ({\bf c}) \in \z$ by \[L_k ({\bf c}) \coloneqq -\langle \lambda, h_{i_k} \rangle + \sum_{1 \le l \le d_{i_k}} (\mu^{(i_k, k-1)} _l({\bf c}) +\nu^{(i_k, k-1)} _l({\bf c})) + \sum_{1 \le s \le N;\ i_s \neq i_k} c_{i_k, i_s} c_s.\] 

\vspace{2mm}\begin{lem}\label{l:nonnegativity}
The integer $L_k ({\bf c})$ is nonnegative. 
\end{lem}

\begin{proof}
Let $b_{\rm high} \in \mathcal{B}^{w_{\ge k}} (\lambda) \cap \Psi_{\bf i} ^{-1}({\bf c} + \z^{d_{i_k}})$ be the unique element such that $\Psi_{\bf i}(b_{\rm high}) = {\bf c} + \mu^{(i_k, k-1)} ({\bf c})$. Then, we see that \[{\rm wt}(b_{\rm high}) = \lambda -\sum_{1 \le l \le d_{i_k}} \mu^{(i_k, k-1)} _l({\bf c}) \alpha_{i_k} -\sum_{1 \le s \le N;\ i_s \neq i_k} c_s \alpha_{i_s},\] and hence that 
\begin{align*}
\langle{\rm wt}(b_{\rm high}), h_{i_k}\rangle &= \langle \lambda, h_{i_k} \rangle -2\sum_{1 \le l \le d_{i_k}} \mu^{(i_k, k-1)} _l({\bf c}) -\sum_{1 \le s \le N;\ i_s \neq i_k} c_{i_k, i_s} c_s.
\end{align*}
From this, it follows that \[L_k({\bf c}) = -\langle{\rm wt}(b_{\rm high}), h_{i_k}\rangle + \sum_{1 \le l \le d_{i_k}} (\nu^{(i_k, k-1)} _l({\bf c}) -\mu^{(i_k, k-1)} _l({\bf c})).\] Since we have $\tilde{f}_{i_k} ^{\varphi_{i_k}(b_{\rm high})} b_{\rm high} \in \mathcal{B}^{w_{\ge k}} (\lambda) \cap \Psi_{\bf i} ^{-1} ({\bf c} + \z^{d_{i_k}})$ by Proposition \ref{p:properties of Demazure} (1), the equality
\begin{align*}
\Psi_{\bf i} (\mathcal{B}^{w_{\ge k}} (\lambda)) \cap ({\bf c} + \z^{d_{i_k}}) = {\bf c} + \Pi_\z(\mu^{(i_k, k-1)} ({\bf c}), \nu^{(i_k, k-1)} ({\bf c}))
\end{align*}
implies that $\varphi_{i_k}(b_{\rm high}) \le \sum_{1 \le l \le d_{i_k}} (\nu^{(i_k, k-1)} _l({\bf c}) -\mu^{(i_k, k-1)} _l({\bf c}))$, and hence that \[\langle {\rm wt}(b_{\rm high}), h_{i_k}\rangle = \varphi_{i_k}(b_{\rm high}) - \varepsilon_{i_k}(b_{\rm high}) \le \sum_{1 \le l \le d_{i_k}} (\nu^{(i_k, k-1)} _l({\bf c}) -\mu^{(i_k, k-1)} _l({\bf c})).\] This proves the lemma.
\end{proof}

We set \[\{s_1 ^{(k)} < \cdots < s_{d_{i_k}} ^{(k)}\} \coloneqq \{1 \le s \le N \mid i_s = i_k\}\] for $1 \le k \le N$, and define $1 \le m_k \le d_{i_k}$ by $s^{(k)} _{m_k} = k$. 

\vspace{2mm}\begin{lem}\label{l:computation of parallelepiped}
For ${\bf c} \in (\z^{d_{i_k}})^\perp$, it follows that $\Psi_{\bf i}(\mathcal{B}^{w_{\ge k}}(\lambda)) \cap ({\bf c} + \z^{d_{i_k}}) \neq \emptyset$ if and only if $\Psi_{\bf i}(\mathcal{B}^{w_{\ge k +1}}(\lambda)) \cap ({\bf c} + \z^{d_{i_k}}) \neq \emptyset$. In this case, the following equalities hold$:$
\begin{align*}
&\mu_l ^{(i_k, k)}({\bf c}) = \mu_l ^{(i_k, k-1)}({\bf c}),\ \nu_l ^{(i_k, k)}({\bf c}) = \nu_l ^{(i_k, k-1)}({\bf c})\ {\it for}\ 1 \le l < m_k,\\ 
&\mu_{m_k} ^{(i_k, k)}({\bf c}) = x_k -L_k({\bf c}),\ \nu_{m_k} ^{(i_k, k)}({\bf c}) = x_k,\ {\it and}\\
&\mu_l ^{(i_k, k)}({\bf c}) = \nu_l ^{(i_k, k)}({\bf c}) = x_{s_l ^{(k)}}\ {\it for}\ m_k < l \le d_{i_k}.
\end{align*}
\end{lem}

\begin{proof}
Since
\begin{equation}\label{e:recursion of opposite 1}
\begin{aligned}
&\mathcal{B}^{w_{\ge k+1}} (\lambda) = \bigcup_{a \ge 0} \tilde{e}_{i_k} ^a \mathcal{B}^{w_{\ge k}} (\lambda) \setminus \{0\}
\end{aligned}
\end{equation}
by Proposition \ref{p:properties of Demazure} (2), the first assertion follows immediately by the crystal structure on $\z^\infty _{\bf j} \otimes R_\lambda$. By Proposition \ref{p:polytopes for opposite}, we have $\mu_l ^{(i_k, k)}({\bf c}) = \nu_l ^{(i_k, k)}({\bf c}) = x_{s_l ^{(k)}}$ for $m_k < l \le d_{i_k}$, and
\begin{equation}\label{e:recursion of opposite 2}
\begin{aligned}
&\Psi_{\bf i}(\mathcal{B}^{w_{\ge k}} (\lambda)) = \{{\bf a} \in \Psi_{\bf i}(\mathcal{B}^{w_{\ge k+1}} (\lambda)) \mid a_k = x_k\}.
\end{aligned}
\end{equation}
By \eqref{e:recursion of opposite 1} and \eqref{e:recursion of opposite 2}, there exists $\widetilde{L}_k({\bf c}) \in \z_{\ge 0}$ such that 
\begin{align*}
&\mu_l ^{(i_k, k)}({\bf c}) = \mu_l ^{(i_k, k-1)}({\bf c}),\ \nu_l ^{(i_k, k)}({\bf c}) = \nu_l ^{(i_k, k-1)}({\bf c})\ {\rm for}\ 1 \le l < m_k,\ {\rm and}\\ 
&\mu_{m_k} ^{(i_k, k)}({\bf c}) = x_k -\widetilde{L}_k({\bf c}),\ \nu_{m_k} ^{(i_k, k)}({\bf c}) = x_k.
\end{align*}
Hence for the second assertion of the lemma, it suffices to show that $\widetilde{L}_k({\bf c}) = L_k({\bf c})$. For $i \in I$, let us consider the Demazure operator $D_i \colon \z[P] \rightarrow \z[P]$ given by \[D_i (e^\lambda) \coloneqq \frac{e^\lambda -e^{s_i(\lambda) +\alpha_i}}{1-e^{\alpha_i}}\] for $\lambda \in P$. For $\lambda \in P$ with $\langle\lambda, h_i\rangle \le 0$, we have \[D_i(e^\lambda) = e^\lambda + e^{\lambda +\alpha_i} + \cdots + e^{s_i(\lambda)}.\] By the string property of $\mathcal{B}^{w_{\ge k}} (\lambda)$ (Proposition \ref{p:string property} (2)) and the equality \[\mathcal{B}^{w_{\ge k+1}} (\lambda) \cap \Psi_{\bf i} ^{-1}({\bf c} + \z^{d_{i_k}}) = \bigcup_{a \ge 0} \tilde{e}_{i_k} ^a (\mathcal{B}^{w_{\ge k}} (\lambda) \cap \Psi_{\bf i} ^{-1}({\bf c} + \z^{d_{i_k}})) \setminus \{0\},\] we deduce that 
\begin{equation}\label{character 1}
\begin{aligned}
{\rm ch}(\mathcal{B}^{w_{\ge k+1}} (\lambda) \cap \Psi_{\bf i} ^{-1}({\bf c} + \z^{d_{i_k}})) = D_{i_k}({\rm ch}(\mathcal{B}^{w_{\ge k}} (\lambda) \cap \Psi_{\bf i} ^{-1}({\bf c} + \z^{d_{i_k}}))). 
\end{aligned}
\end{equation}
Set $\Pi_1 \coloneqq \Pi_\z (\mu^{(i_k, k)}({\bf c}), \nu^{(i_k, k)}({\bf c}))$ and $\Pi_2 \coloneqq \Pi_\z (\hat{\mu}^{(i_k, k)}({\bf c}), \nu^{(i_k, k)}({\bf c}))$, where we define $\hat{\mu}^{(i_k, k)}({\bf c})$ by replacing $\mu_{m_k} ^{(i_k, k-1)}({\bf c}) = x_k$ in $\mu^{(i_k, k-1)}({\bf c})$ by $x_k -L_k({\bf c})$. Then, it follows that 
\begin{equation}\label{character 2}
\begin{aligned}
D_{i_k}\left({\rm ch}(\mathcal{B}^{w_{\ge k}} (\lambda) \cap \Psi_{\bf i} ^{-1}({\bf c} + \z^{d_{i_k}}))\right) = e^{\lambda -\sum_{1 \le s \le N;\ i_s \neq i_k} c_s \alpha_{i_s}} \sum_{(a_1 ^{(i)}, \ldots, a_{d_{i_k}} ^{(i)}) \in \Pi_2} e^{(a^{(i)} _1 + \cdots + a^{(i)} _{d_{i_k}}) \alpha_{i_k}};
\end{aligned}
\end{equation}
see \cite[Proposition 6.3]{KST}. From the equalities \eqref{character 1} and \eqref{character 2}, we see that 
\begin{align*}
\sum_{(a_1 ^{(i)}, \ldots, a_{d_{i_k}} ^{(i)}) \in \Pi_1} e^{(a^{(i)} _1 + \cdots + a^{(i)} _{d_{i_k}}) \alpha_{i_k}} = \sum_{(a_1 ^{(i)}, \ldots, a_{d_{i_k}} ^{(i)}) \in \Pi_2} e^{(a^{(i)} _1 + \cdots + a^{(i)} _{d_{i_k}}) \alpha_{i_k}}.
\end{align*}
By comparing the number of terms, we deduce that $\widetilde{L}_k({\bf c}) = L_k({\bf c})$. This proves the lemma.
\end{proof}

For subsets $X, Y \subset \r^N$, we define $X + Y$ to be the Minkowski sum: \[X + Y \coloneqq \{x + y \mid x \in X,\ y \in Y\}.\]

\vspace{2mm}\begin{lem}\label{l:additivity of lattice points}
Let ${\bf i} = (i_1, \ldots, i_N) \in I^N$ be a reduced word for $w_0$, and $\lambda_1, \lambda_2 \in P_+$. Assume that the polytopes $\Delta_{\bf i}(\lambda_1), \Delta_{\bf i}(\lambda_2)$, and $\Delta_{\bf i}(\lambda_1 + \lambda_2)$ are all parapolytopes. Then, the following equality holds for all $1 \le k \le N+1$$:$ \[\Psi_{\bf i}(\mathcal{B}^{w_{\ge k}}(\lambda_1 + \lambda_2)) = \Psi_{\bf i}(\mathcal{B}^{w_{\ge k}}(\lambda_1)) + \Psi_{\bf i}(\mathcal{B}^{w_{\ge k}}(\lambda_2)).\]
\end{lem}

\begin{proof}
We proceed by induction on $k$. If $k = 1$, then the assertion is obvious since \[\Psi_{\bf i} (b_{w_0 (\lambda_1 + \lambda_2)}) = \Psi_{\bf i} (b_{w_0 \lambda_1}) + \Psi_{\bf i} (b_{w_0 \lambda_2})\] by Theorem \ref{t:Nakashima}. Let $1 \le k \le N$, and assume that 
\begin{equation}\label{inductive hypothesis for additivity}
\begin{aligned}
\Psi_{\bf i}(\mathcal{B}^{w_{\ge k}}(\lambda_1 + \lambda_2)) = \Psi_{\bf i}(\mathcal{B}^{w_{\ge k}}(\lambda_1)) + \Psi_{\bf i}(\mathcal{B}^{w_{\ge k}}(\lambda_2)). 
\end{aligned}
\end{equation}
By Corollary \ref{c:additivity of opposite}, for the inductive step, it suffices to prove that \[\Psi_{\bf i}(\mathcal{B}^{w_{\ge k+1}}(\lambda_1 + \lambda_2)) \subset \Psi_{\bf i}(\mathcal{B}^{w_{\ge k+1}}(\lambda_1)) + \Psi_{\bf i}(\mathcal{B}^{w_{\ge k+1}}(\lambda_2)).\] Fix ${\bf c} \in (\z^{d_{i_k}})^\perp$ such that $\Psi_{\bf i}(\mathcal{B}^{w_{\ge k}}(\lambda_1 + \lambda_2)) \cap ({\bf c} + \z^{d_{i_k}}) \neq \emptyset$. We denote $\mu^{(i_k, k-1)} ({\bf c}), \nu^{(i_k, k-1)} ({\bf c}), L_k ({\bf c})$ for $\mathcal{B}^{w_{\ge k}} (\lambda)$ by $\mu^{(i_k, k-1)} (\lambda, {\bf c}), \nu^{(i_k, k-1)} (\lambda, {\bf c}), L_k (\lambda, {\bf c})$, respectively, where $\lambda = \lambda_1, \lambda_2, \lambda_1 +\lambda_2$. The equality \eqref{inductive hypothesis for additivity} implies that 
\begin{align*}
&\Psi_{\bf i}(\mathcal{B}^{w_{\ge k}}(\lambda_1 + \lambda_2)) \cap ({\bf c} + \z^{d_{i_k}})\\ 
=\ &\bigcup_{{\bf c}_1, {\bf c}_2 \in (\z^{d_{i_k}})^\perp;\ {\bf c}_1 + {\bf c}_2 = {\bf c}} \left(\Psi_{\bf i}(\mathcal{B}^{w_{\ge k}}(\lambda_1)) \cap ({\bf c}_1 + \z^{d_{i_k}}) + \Psi_{\bf i}(\mathcal{B}^{w_{\ge k}}(\lambda_2)) \cap ({\bf c}_2 + \z^{d_{i_k}})\right), 
\end{align*}
and hence that 
\begin{align*}
&\Pi_\z(\mu^{(i_k, k-1)} (\lambda_1 + \lambda_2, {\bf c}), \nu^{(i_k, k-1)} (\lambda_1 + \lambda_2, {\bf c}))\\ 
=\ &\bigcup_{{\bf c}_1, {\bf c}_2 \in (\z^{d_{i_k}})^\perp;\ {\bf c}_1 + {\bf c}_2 = {\bf c}} \Pi_\z(\mu^{(i_k, k-1)} (\lambda_1, {\bf c}_1) +\mu^{(i_k, k-1)} (\lambda_2, {\bf c}_2), \nu^{(i_k, k-1)} (\lambda_1, {\bf c}_1) +\nu^{(i_k, k-1)} (\lambda_2, {\bf c}_2)). 
\end{align*}
From this, there exist ${\bf c}_1, {\bf c}_2 \in (\z^{d_{i_k}})^\perp$ such that ${\bf c}_1 + {\bf c}_2 = {\bf c}$, and such that 
\begin{equation}\label{comparison of edges}
\begin{aligned}
&\nu^{(i_k, k-1)} (\lambda_1 + \lambda_2, {\bf c}) = \nu^{(i_k, k-1)} (\lambda_1, {\bf c}_1) +\nu^{(i_k, k-1)} (\lambda_2, {\bf c}_2),\\
&\mu^{(i_k, k-1)} _l(\lambda_1 + \lambda_2, {\bf c}) \le \mu^{(i_k, k-1)} _l(\lambda_1, {\bf c}_1) +\mu^{(i_k, k-1)} _l(\lambda_2, {\bf c}_2)\ {\rm for\ all}\ 1 \le l \le d_{i_k}.
\end{aligned}
\end{equation}
Since \[\Psi_{\bf i}(\mathcal{B}^{w_{\ge k+1}}(\lambda_1)) + \Psi_{\bf i}(\mathcal{B}^{w_{\ge k+1}}(\lambda_2)) \subset \Psi_{\bf i}(\mathcal{B}^{w_{\ge k+1}}(\lambda_1 + \lambda_2))\] by Corollary \ref{c:additivity of opposite}, we have 
\begin{equation}\label{inclusion of parallelepiped for additivity}
\begin{aligned}
&\Pi_\z(\mu^{(i_k, k)} (\lambda_1, {\bf c}_1) +\mu^{(i_k, k)} (\lambda_2, {\bf c}_2), \nu^{(i_k, k)} (\lambda_1, {\bf c}_1) +\nu^{(i_k, k)} (\lambda_2, {\bf c}_2))\\ 
\subset\ &\Pi_\z(\mu^{(i_k, k)} (\lambda_1 + \lambda_2, {\bf c}), \nu^{(i_k, k)} (\lambda_1 + \lambda_2, {\bf c})).
\end{aligned}
\end{equation}
Also, Lemma \ref{l:computation of parallelepiped} implies that 
\begin{equation}\label{eq:additivity of nu for k}
\begin{aligned}
\nu^{(i_k, k)} (\lambda_1 + \lambda_2, {\bf c}) &= \nu^{(i_k, k-1)} (\lambda_1 + \lambda_2, {\bf c})\\
&= \nu^{(i_k, k-1)} (\lambda_1, {\bf c}_1) + \nu^{(i_k, k-1)} (\lambda_2, {\bf c}_2)\\
&= \nu^{(i_k, k)} (\lambda_1, {\bf c}_1) + \nu^{(i_k, k)} (\lambda_2, {\bf c}_2),
\end{aligned}
\end{equation}
and that 
\begin{align*}
\mu_{m_k} ^{(i_k, k)} (\lambda_1 + \lambda_2, {\bf c}) &= -\langle w_{\ge k} (\lambda_1 + \lambda_2), h_{i_k} \rangle -L_k(\lambda_1 + \lambda_2, {\bf c})\\
&\ge-\langle w_{\ge k} \lambda_1, h_{i_k} \rangle -\langle w_{\ge k} \lambda_2, h_{i_k} \rangle -(L_k(\lambda_1, {\bf c}_1) +L_k(\lambda_2, {\bf c}_2))\\
&({\rm by}\ \eqref{comparison of edges}\ {\rm and\ the\ definition\ of}\ L_k({\bf c}))\\
&= \mu_{m_k} ^{(i_k, k)} (\lambda_1, {\bf c}_1) +\mu_{m_k} ^{(i_k, k)} (\lambda_2, {\bf c}_2).
\end{align*}
In addition, this inequality becomes the equality if and only if $\mu_l ^{(i_k, k-1)} (\lambda_1 + \lambda_2, {\bf c}) = \mu_l ^{(i_k, k-1)} (\lambda_1, {\bf c}_1) +\mu_l ^{(i_k, k-1)} (\lambda_2, {\bf c}_2)$ for all $1 \le l \le d_{i_k}$. However, the inclusion relation \eqref{inclusion of parallelepiped for additivity} implies that $\mu_{m_k} ^{(i_k, k)} (\lambda_1, {\bf c}_1) +\mu_{m_k} ^{(i_k, k)} (\lambda_2, {\bf c}_2) \ge \mu_{m_k} ^{(i_k, k)} (\lambda_1 + \lambda_2, {\bf c})$, and hence that $\mu_{m_k} ^{(i_k, k)} (\lambda_1, {\bf c}_1) +\mu_{m_k} ^{(i_k, k)} (\lambda_2, {\bf c}_2) = \mu_{m_k} ^{(i_k, k)} (\lambda_1 + \lambda_2, {\bf c})$. This proves $\mu^{(i_k, k)} (\lambda_1 + \lambda_2, {\bf c}) = \mu^{(i_k, k)} (\lambda_1, {\bf c}_1) +\mu^{(i_k, k)} (\lambda_2, {\bf c}_2)$ by Lemma \ref{l:computation of parallelepiped}, which implies by \eqref{eq:additivity of nu for k} that 
\begin{align*}
&{\bf c} +\Pi_\z(\mu^{(i_k, k)} (\lambda_1 + \lambda_2, {\bf c}), \nu^{(i_k, k)} (\lambda_1 + \lambda_2, {\bf c}))\\ 
=\ &\left({\bf c}_1 +\Pi_\z(\mu^{(i_k, k)} (\lambda_1, {\bf c}_1), \nu^{(i_k, k)} (\lambda_1, {\bf c}_1))\right) + \left({\bf c}_2 +\Pi_\z(\mu^{(i_k, k)} (\lambda_2, {\bf c}_2), \nu^{(i_k, k)} (\lambda_2, {\bf c}_2))\right)\\
\subset\ &\Psi_{\bf i}(\mathcal{B}^{w_{\ge k+1}}(\lambda_1)) + \Psi_{\bf i}(\mathcal{B}^{w_{\ge k+1}}(\lambda_2)).
\end{align*}
Hence we conclude that $\Psi_{\bf i}(\mathcal{B}^{w_{\ge k+1}}(\lambda_1 + \lambda_2)) \subset \Psi_{\bf i}(\mathcal{B}^{w_{\ge k+1}}(\lambda_1)) + \Psi_{\bf i}(\mathcal{B}^{w_{\ge k+1}}(\lambda_2))$. This proves the lemma. 
\end{proof}

Note that $\Delta_{\bf i}(m\lambda) = m\Delta_{\bf i}(\lambda)$ for $m \in \z_{>0}$ by the additivity of $\Psi_{\bf i}$ (see \cite[Theorem 4.1]{FN}). Hence if $\Delta_{\bf i}(\lambda)$ is a parapolytope, then the polytopes $\Delta_{\bf i}(m\lambda)$, $m \in \z_{>0}$, are all parapolytopes. By Lemma \ref{l:additivity of lattice points}, this implies that \[\Psi_{\bf i}(\mathcal{B}(m\lambda)) = \underbrace{\Psi_{\bf i}(\mathcal{B}(\lambda)) + \cdots + \Psi_{\bf i}(\mathcal{B}(\lambda))}_m\] for $m \in \z_{>0}$, and hence that the equality $\Delta_{\bf i}(\lambda) = {\rm Conv}(\Psi_{\bf i}(\mathcal{B}(\lambda)))$ holds. This proves part (1) of Theorem \ref{t:main result}.

For ${\bf c} = (c_s)_{1 \le s \le N;\ i_s \neq i_k} \in (\r^{d_{i_k}})^\perp$ and $l = k-1, k$ such that ${\rm Conv}(\Psi_{\bf i}(\mathcal{B}^{w_{\ge l +1}}(\lambda))) \cap ({\bf c} + \r^{d_{i_k}}) \neq \emptyset$, there uniquely exist \[\mu_+ ^{(i_k, l)} ({\bf c}) = (\mu_{+, 1} ^{(i_k, l)} ({\bf c}), \ldots, \mu_{+, d_{i_k}} ^{(i_k, l)} ({\bf c})),\ \nu_+ ^{(i_k, l)} ({\bf c}) = (\nu_{+, 1} ^{(i_k, l)} ({\bf c}), \ldots, \nu_{+, d_{i_k}} ^{(i_k, l)} ({\bf c})) \in \r^{d_{i_k}}\] such that \[{\rm Conv}(\Psi_{\bf i}(\mathcal{B}^{w_{\ge l +1}}(\lambda))) \cap ({\bf c} + \r^{d_{i_k}}) = {\bf c} + \Pi(\mu_+ ^{(i_k, l)} ({\bf c}), \nu_+ ^{(i_k, l)} ({\bf c})).\] If we set 
\begin{align*}
&\tilde{\bf c} \coloneqq (\hat{\lambda}_{i_s} -c_s)_{1 \le s \le N;\ i_s \neq i_k},\\ 
&\tilde{\mu}_+ ^{(i_k, l)}({\bf c}) = (\tilde{\mu}_{+, 1} ^{(i_k, l)}({\bf c}), \ldots, \tilde{\mu}_{+, d_{i_k}} ^{(i_k, l)}({\bf c})) \coloneqq (\hat{\lambda}_{i_k}, \ldots, \hat{\lambda}_{i_k}) -\nu_+ ^{(i_k, l)}({\bf c}),\ {\rm and}\\
&\tilde{\nu}_+ ^{(i_k, l)}({\bf c}) = (\tilde{\nu}_{+, 1} ^{(i_k, l)}({\bf c}), \ldots, \tilde{\nu}_{+, d_{i_k}} ^{(i_k, l)}({\bf c})) \coloneqq (\hat{\lambda}_{i_k}, \ldots, \hat{\lambda}_{i_k}) -\mu_+ ^{(i_k, l)}({\bf c})
\end{align*}
for $l = k-1, k$, then we have \[\left(-{\rm Conv}(\Psi_{\bf i}(\mathcal{B}^{w_{\ge l +1}}(\lambda))) + (\hat{\lambda}_{i_1}, \ldots, \hat{\lambda}_{i_N})\right) \cap (\tilde{\bf c} + \r^{d_{i_k}}) = \tilde{\bf c} + \Pi(\tilde{\mu}_+ ^{(i_k, l)}({\bf c}), \tilde{\nu}_+ ^{(i_k, l)}({\bf c})).\] 

\vspace{2mm}\begin{lem}\label{l:domain of mu and nu}
For ${\bf c} \in (\r^{d_{i_k}})^\perp$, it follows that ${\rm Conv}(\Psi_{\bf i}(\mathcal{B}^{w_{\ge k}}(\lambda))) \cap ({\bf c} + \r^{d_{i_k}}) \neq \emptyset$ if and only if ${\rm Conv}(\Psi_{\bf i}(\mathcal{B}^{w_{\ge k +1}}(\lambda))) \cap ({\bf c} + \r^{d_{i_k}}) \neq \emptyset$. 
\end{lem}

\begin{proof}
If we denote by $P_k \colon \r^N \rightarrow (\r^{d_{i_k}})^\perp$ the canonical projection, then we have 
\begin{align*}
P_k ({\rm Conv}(\Psi_{\bf i}(\mathcal{B}^{w_{\ge l}}(\lambda)))) = \{{\bf c} \in (\r^{d_{i_k}})^\perp \mid {\rm Conv}(\Psi_{\bf i}(\mathcal{B}^{w_{\ge l}}(\lambda))) \cap ({\bf c} + \r^{d_{i_k}}) \neq \emptyset\}
\end{align*}
for $l = k, k+1$. Hence it suffices to prove that $P_k ({\rm Conv}(\Psi_{\bf i}(\mathcal{B}^{w_{\ge k +1}}(\lambda)))) = P_k ({\rm Conv}(\Psi_{\bf i}(\mathcal{B}^{w_{\ge k}}(\lambda))))$. 

Since $\mathcal{B}^{w_{\ge k}}(\lambda) \subset \mathcal{B}^{w_{\ge k +1}}(\lambda)$, we have $P_k ({\rm Conv}(\Psi_{\bf i}(\mathcal{B}^{w_{\ge k}}(\lambda)))) \subset P_k ({\rm Conv}(\Psi_{\bf i}(\mathcal{B}^{w_{\ge k +1}}(\lambda))))$. Let ${\bf c} \in (\z^{d_{i_k}})^\perp$ be a vertex of the lattice polytope $P_k ({\rm Conv}(\Psi_{\bf i}(\mathcal{B}^{w_{\ge k +1}}(\lambda))))$. Then, it follows that ${\rm Conv}(\Psi_{\bf i}(\mathcal{B}^{w_{\ge k +1}}(\lambda))) \cap ({\bf c} + \r^{d_{i_k}}) \neq \emptyset$, and that $\mu_+ ^{(i_k, k)} ({\bf c}), \nu_+ ^{(i_k, k)} ({\bf c}) \in \q^{d_{i_k}}$. We take $l \in \z_{>0}$ such that $l \mu_+ ^{(i_k, k)} ({\bf c}), l \nu_+ ^{(i_k, k)} ({\bf c}) \in \z^{d_{i_k}}$. Since we have
\begin{align*}
{\rm Conv}(\Psi_{\bf i}(\mathcal{B}^{w_{\ge k +1}}(l\lambda))) \cap (l{\bf c} + \r^{d_{i_k}}) &= l \left({\rm Conv}(\Psi_{\bf i}(\mathcal{B}^{w_{\ge k +1}}(\lambda))) \cap ({\bf c} + \r^{d_{i_k}})\right)\\
&= l{\bf c} + \Pi(l \mu_+ ^{(i_k, k)} ({\bf c}), l \nu_+ ^{(i_k, k)} ({\bf c})),
\end{align*}
it follows that \[\Psi_{\bf i}(\mathcal{B}^{w_{\ge k +1}}(l\lambda)) \cap (l{\bf c} + \z^{d_{i_k}}) = \left({\rm Conv}(\Psi_{\bf i}(\mathcal{B}^{w_{\ge k +1}}(l\lambda))) \cap (l{\bf c} + \r^{d_{i_k}})\right) \cap \z^N \neq \emptyset.\] Hence Lemma \ref{l:computation of parallelepiped} implies that \[\Psi_{\bf i}(\mathcal{B}^{w_{\ge k}}(l\lambda)) \cap (l{\bf c} + \z^{d_{i_k}}) \neq \emptyset,\] and hence that \[{\rm Conv}(\Psi_{\bf i}(\mathcal{B}^{w_{\ge k}}(\lambda))) \cap ({\bf c} + \r^{d_{i_k}}) = \frac{1}{l} \left({\rm Conv}(\Psi_{\bf i}(\mathcal{B}^{w_{\ge k}}(l \lambda))) \cap (l {\bf c} + \r^{d_{i_k}})\right) \neq \emptyset.\] Thus, the vertices of $P_k ({\rm Conv}(\Psi_{\bf i}(\mathcal{B}^{w_{\ge k +1}}(\lambda))))$ are contained in $P_k ({\rm Conv}(\Psi_{\bf i}(\mathcal{B}^{w_{\ge k}}(\lambda))))$. From this and the convexity of $P_k ({\rm Conv}(\Psi_{\bf i}(\mathcal{B}^{w_{\ge k}}(\lambda))))$, we obtain $P_k ({\rm Conv}(\Psi_{\bf i}(\mathcal{B}^{w_{\ge k +1}}(\lambda)))) \subset P_k ({\rm Conv}(\Psi_{\bf i}(\mathcal{B}^{w_{\ge k}}(\lambda))))$, which proves the lemma. 
\end{proof}

\begin{lem}\label{l:divided difference}
The polytope ${\rm Conv}(\Psi_{\bf i}(\mathcal{B}^{w_{\ge k}}(\lambda)))$ is a parapolytope for all $1 \le k \le N +1$, and the following equality holds for all $1 \le k \le N$$:$ \[-{\rm Conv}(\Psi_{\bf i}(\mathcal{B}^{w_{\ge k+1}}(\lambda))) + (\hat{\lambda}_{i_1}, \ldots, \hat{\lambda}_{i_N}) = D_{i_k} ^{(k)} \left(-{\rm Conv}(\Psi_{\bf i}(\mathcal{B}^{w_{\ge k}}(\lambda))) + (\hat{\lambda}_{i_1}, \ldots, \hat{\lambda}_{i_N})\right).\]
\end{lem}

\begin{proof}
Since $\Delta_{\bf i}(\lambda) = {\rm Conv}(\Psi_{\bf i}(\mathcal{B}(\lambda)))$, Proposition \ref{p:polytopes for opposite} implies that \[{\rm Conv}(\Psi_{\bf i}(\mathcal{B}^{w_{\ge k}}(\lambda))) = \{(a_1, \ldots, a_N) \in \Delta_{\bf i}(\lambda) \mid a_k = x_k, \ldots, a_N = x_N\},\] and hence that this is a parapolytope. In particular, a function \[D_{i_k} ^{(k)} \left(-{\rm Conv}(\Psi_{\bf i}(\mathcal{B}^{w_{\ge k}}(\lambda))) + (\hat{\lambda}_{i_1}, \ldots, \hat{\lambda}_{i_N})\right)\] is well-defined. We will show that 
\begin{align*}
&\left(-{\rm Conv}(\Psi_{\bf i}(\mathcal{B}^{w_{\ge k+1}}(\lambda))) + (\hat{\lambda}_{i_1}, \ldots, \hat{\lambda}_{i_N})\right) \cap (\tilde{\bf c} + \r^{d_{i_k}}) \\
=\ &D_{i_k} ^{(k)} \left(\left(-{\rm Conv}(\Psi_{\bf i}(\mathcal{B}^{w_{\ge k}}(\lambda))) + (\hat{\lambda}_{i_1}, \ldots, \hat{\lambda}_{i_N})\right) \cap (\tilde{\bf c} + \r^{d_{i_k}})\right)
\end{align*}
for all ${\bf c} \in (\r^{d_{i_k}})^\perp$ such that ${\rm Conv}(\Psi_{\bf i}(\mathcal{B}^{w_{\ge k}}(\lambda))) \cap ({\bf c} + \r^{d_{i_k}}) \neq \emptyset$. 

First, we consider the case ${\bf c} \in (\q^{d_{i_k}})^\perp$, where we set $(\q^{d_{i_k}})^\perp \coloneqq (\r^{d_{i_k}})^\perp \cap \q^N$. In this case, we have \[\mu_+ ^{(i_k, k-1)} ({\bf c}), \nu_+ ^{(i_k, k-1)} ({\bf c}), \mu_+ ^{(i_k, k)} ({\bf c}), \nu_+ ^{(i_k, k)} ({\bf c}) \in \q^{d_{i_k}}.\] By the definition of $D_{i_k} ^{(k)}$, it suffices to prove that there exists $l \in \z_{> 0}$ such that 
\begin{align*}
&l \left(-{\rm Conv}(\Psi_{\bf i}(\mathcal{B}^{w_{\ge k+1}}(\lambda))) + (\hat{\lambda}_{i_1}, \ldots, \hat{\lambda}_{i_N})\right) \cap (l\tilde{\bf c} + \r^{d_{i_k}}) \\
=\ &D_{i_k} ^{(k)} \left(l\left(-{\rm Conv}(\Psi_{\bf i}(\mathcal{B}^{w_{\ge k}}(\lambda))) + (\hat{\lambda}_{i_1}, \ldots, \hat{\lambda}_{i_N})\right) \cap (l\tilde{\bf c} + \r^{d_{i_k}})\right).
\end{align*}
From this, we may assume that ${\bf c} \in (\z^{d_{i_k}})^\perp,\ \mu_+ ^{(i_k, k-1)} ({\bf c}), \nu_+ ^{(i_k, k-1)} ({\bf c}), \mu_+ ^{(i_k, k)} ({\bf c}), \nu_+ ^{(i_k, k)} ({\bf c}) \in \z^{d_{i_k}}$. Then, the following equalities hold: 
\begin{align*}
&\mu_+ ^{(i_k, k-1)} ({\bf c}) = \mu^{(i_k, k-1)} ({\bf c}),\ \nu_+ ^{(i_k, k-1)} ({\bf c}) = \nu^{(i_k, k-1)} ({\bf c}),\\
&\mu_+ ^{(i_k, k)} ({\bf c}) = \mu^{(i_k, k)} ({\bf c}),\ \nu_+ ^{(i_k, k)} ({\bf c}) = \nu^{(i_k, k)} ({\bf c}).
\end{align*}
We set 
\begin{align*}
&\tilde{\mu}^{(i_k, l)}({\bf c}) = (\tilde{\mu}_{1} ^{(i_k, l)}({\bf c}), \ldots, \tilde{\mu}_{d_{i_k}} ^{(i_k, l)}({\bf c})) \coloneqq (\hat{\lambda}_{i_k}, \ldots, \hat{\lambda}_{i_k}) -\nu^{(i_k, l)}({\bf c}),\ {\rm and}\\
&\tilde{\nu}^{(i_k, l)}({\bf c}) = (\tilde{\nu}_{1} ^{(i_k, l)}({\bf c}), \ldots, \tilde{\nu}_{d_{i_k}} ^{(i_k, l)}({\bf c})) \coloneqq (\hat{\lambda}_{i_k}, \ldots, \hat{\lambda}_{i_k}) -\mu^{(i_k, l)}({\bf c})
\end{align*}
for $l = k-1, k$. By the definition of $D_{i_k} ^{(k)}$, the polytope \[D_{i_k} ^{(k)} \left(\left(-{\rm Conv}(\Psi_{\bf i}(\mathcal{B}^{w_{\ge k}}(\lambda))) + (\hat{\lambda}_{i_1}, \ldots, \hat{\lambda}_{i_N})\right) \cap (\tilde{\bf c} + \r^{d_{i_k}})\right) = D_{i_k} ^{(k)} \left(\tilde{\bf c} + \Pi(\tilde{\mu}^{(i_k, k-1)} ({\bf c}), \tilde{\nu}^{(i_k, k-1)} ({\bf c}))\right)\] is given by replacing $\tilde{\nu}^{(i_k, k-1)} _{m_k}({\bf c})$ in $\tilde{\nu}^{(i_k, k-1)}({\bf c})$ with \[\hat{\nu}^{(i_k, k-1)} _{m_k}({\bf c}) \coloneqq \tilde{\nu}^{(i_k, k-1)} _{m_k}({\bf c}) -\sum_{1 \le s \le N;\ i_s \neq i_k} c_{i_k, i_s}(\hat{\lambda}_{i_s} -c_s) -\sum_{1 \le l \le d_{i_k}} (\tilde{\mu}_l ^{(i_k, k-1)}({\bf c}) + \tilde{\nu}_l ^{(i_k, k-1)}({\bf c}))\] if $\hat{\nu}^{(i_k, k-1)} _{m_k}({\bf c}) \ge \tilde{\nu}^{(i_k, k-1)} _{m_k}({\bf c})$. Note that 
\begin{align*}
\hat{\nu}^{(i_k, k-1)} _{m_k}({\bf c}) -\tilde{\nu}^{(i_k, k-1)} _{m_k}({\bf c}) &= -\sum_{1 \le s \le N} c_{i_k, i_s} \hat{\lambda}_{i_s}  + \sum_{1 \le s \le N;\ i_s \neq i_k} c_{i_k, i_s} c_s + \sum_{1 \le l \le d_{i_k}} (\mu_l ^{(i_k, k-1)}({\bf c}) + \nu_l ^{(i_k, k-1)}({\bf c}))\\ 
&= L_k ({\bf c})
\end{align*}
since $\sum_{1 \le s \le N} c_{i_k, i_s} \hat{\lambda}_{i_s} = \langle\lambda, h_{i_k}\rangle$ by $\lambda = \sum_{i \in I} \hat{\lambda}_i d_i \alpha_i$. Since $L_k ({\bf c}) \ge 0$ by Lemma \ref{l:nonnegativity}, it follows that $D_{i_k} ^{(k)} (\tilde{\bf c} + \Pi(\tilde{\mu}^{(i_k, k-1)}({\bf c}), \tilde{\nu}^{(i_k, k-1)}({\bf c})))$ is the polytope given by replacing $\tilde{\nu}^{(i_k, k-1)} _{m_k}({\bf c})$ in $\tilde{\bf c} + \Pi(\tilde{\mu}^{(i_k, k-1)}({\bf c}), \tilde{\nu}^{(i_k, k-1)}({\bf c}))$ with $\hat{\nu}^{(i_k, k-1)} _{m_k}({\bf c}) = \tilde{\nu}^{(i_k, k-1)} _{m_k}({\bf c}) + L_k({\bf c})$, which implies by Lemma \ref{l:computation of parallelepiped} that \[D_{i_k} ^{(k)} (\tilde{\bf c} + \Pi(\tilde{\mu}^{(i_k, k-1)} ({\bf c}), \tilde{\nu}^{(i_k, k-1)} ({\bf c}))) = \tilde{\bf c} + \Pi(\tilde{\mu}^{(i_k, k)} ({\bf c}), \tilde{\nu}^{(i_k, k)} ({\bf c})).\]

Second, we consider the case ${\bf c} \in (\r^{d_{i_k}})^\perp$. We regard $\tilde{\mu}_{+, l} ^{(i_k, k-1)} ({\bf c}), \tilde{\nu}_{+, l} ^{(i_k, k-1)} ({\bf c}), \tilde{\mu}_{+, l} ^{(i_k, k)} ({\bf c}), \tilde{\nu}_{+, l} ^{(i_k, k)} ({\bf c})$ for $1 \le l \le d_{i_k}$ as $\r$-valued functions on the lattice polytope \[P_k ({\rm Conv}(\Psi_{\bf i}(\mathcal{B}^{w_{\ge k+1}}(\lambda)))) = P_k ({\rm Conv}(\Psi_{\bf i}(\mathcal{B}^{w_{\ge k}}(\lambda))));\] see the proof of Lemma \ref{l:domain of mu and nu}. Since $-{\rm Conv}(\Psi_{\bf i}(\mathcal{B}^{w_{\ge k +1}}(\lambda))) + (\hat{\lambda}_{i_1}, \ldots, \hat{\lambda}_{i_N})$ and $-{\rm Conv}(\Psi_{\bf i}(\mathcal{B}^{w_{\ge k}}(\lambda))) + (\hat{\lambda}_{i_1}, \ldots, \hat{\lambda}_{i_N})$ are convex, the functions $\tilde{\mu}_{+, l} ^{(i_k, k-1)} ({\bf c}), \tilde{\nu}_{+, l} ^{(i_k, k-1)} ({\bf c}), \tilde{\mu}_{+, l} ^{(i_k, k)} ({\bf c}), \tilde{\nu}_{+, l} ^{(i_k, k)} ({\bf c})$ are (upper or lower) convex on each line segment $S \subset P_k ({\rm Conv}(\Psi_{\bf i}(\mathcal{B}^{w_{\ge k+1}}(\lambda))))$; hence they are continuous on the relative interior of $S$. From this and the assertion in the case ${\bf c} \in (\q^{d_{i_k}})^\perp$, we deduce that 
\begin{align*}
&\tilde{\nu}^{(i_k, k-1)} _{m_k}({\bf c}) \le \hat{\nu}^{(i_k, k-1)} _{m_k}({\bf c}),\\
&D_{i_k} ^{(k)} (\tilde{\bf c} + \Pi(\tilde{\mu}_+ ^{(i_k, k-1)} ({\bf c}), \tilde{\nu}_+ ^{(i_k, k-1)} ({\bf c}))) = \tilde{\bf c} + \Pi(\tilde{\mu}_+ ^{(i_k, k)} ({\bf c}), \tilde{\nu}_+ ^{(i_k, k)} ({\bf c}))
\end{align*}
for all ${\bf c} \in P_k ({\rm Conv}(\Psi_{\bf i}(\mathcal{B}^{w_{\ge k+1}}(\lambda))))$. This proves the lemma.
\end{proof}

Since we have 
\begin{align*}
-{\rm Conv}(\Psi_{\bf i}(\mathcal{B}^{w_{\ge 1}}(\lambda))) + (\hat{\lambda}_{i_1}, \ldots, \hat{\lambda}_{i_N}) &= -{\bf x}_\lambda + (\hat{\lambda}_{i_1}, \ldots, \hat{\lambda}_{i_N})\\
&= {\bf a}_\lambda,
\end{align*}
Lemma \ref{l:divided difference} implies that $D_{i_k} ^{(k)} \cdots D_{i_1} ^{(1)} ({\bf a}_\lambda)$ is a well-defined parapolytope for $1 \le k \le N$, and that the following equality holds for $1 \le k \le N$$:$ \[D_{i_k} ^{(k)} \cdots D_{i_1} ^{(1)} ({\bf a}_\lambda) = -{\rm Conv}(\Psi_{\bf i}(\mathcal{B}^{w_{\ge k +1}}(\lambda))) + (\hat{\lambda}_{i_1}, \ldots, \hat{\lambda}_{i_N}).\] From these, we obtain parts (2), (3) of Theorem \ref{t:main result}. 

\subsection{Immediate consequences}

By Theorem \ref{t:main result} (1) and Lemma \ref{l:additivity of lattice points}, we obtain the following.

\vspace{2mm}\begin{thm}\label{t:corollary 1}
Let ${\bf i} \in I^N$ be a reduced word for $w_0$, and $\lambda, \mu \in P_+$. Assume that the polytopes $\Delta_{\bf i}(\lambda), \Delta_{\bf i}(\mu)$, and $\Delta_{\bf i}(\lambda + \mu)$ are all parapolytopes. Then, the following equalities hold$:$ 
\begin{align*}
&\Psi_{\bf i}(\mathcal{B}(\lambda + \mu)) = \Psi_{\bf i}(\mathcal{B}(\lambda)) + \Psi_{\bf i}(\mathcal{B}(\mu)),\ {\it and}\\
&\Delta_{\bf i}(\lambda + \mu) = \Delta_{\bf i}(\lambda) + \Delta_{\bf i}(\mu).
\end{align*}
\end{thm}\vspace{2mm}

The proof of Theorem \ref{t:main result} implies the following.

\vspace{2mm}\begin{prop}\label{c:corollary 2}
Let ${\bf i} = (i_1, \ldots, i_N) \in I^N$ be a reduced word for $w_0$, $\lambda \in P_+$, and $2 \le k \le N$. Assume that the face $\{{\bf a} \in \Delta_{\bf i}(\lambda) \mid a_k = x_k, \ldots, a_N = x_N\}$ of $\Delta_{\bf i}(\lambda)$ is a parapolytope. 
\begin{enumerate}
\item[{\rm (1)}] The face $\{{\bf a} \in \Delta_{\bf i}(\lambda) \mid a_k = x_k, \ldots, a_N = x_N\}$ is a lattice polytope.
\item[{\rm (2)}] The polytope $D_{i_l} ^{(l)} \cdots D_{i_1} ^{(1)} ({\bf a}_\lambda)$ is well-defined for $1 \le l \le k-1$.
\item[{\rm (3)}] The following equality holds for all $1 \le l \le k-1$$:$ \[D_{i_l} ^{(l)} \cdots D_{i_1} ^{(1)} ({\bf a}_\lambda) = -\{{\bf a} \in \Delta_{\bf i}(\lambda) \mid a_{l+1} = x_{l+1}, \ldots, a_N = x_N\} + (\hat{\lambda}_{i_1}, \ldots, \hat{\lambda}_{i_N}).\] 
\end{enumerate}
\end{prop}\vspace{2mm}

\section{Crystal structures}

In this section, we study the crystal structure on the set of lattice points in $\Delta_{\bf i} (\lambda)$. Recall that $e_i, f_i, h_i \in \mathfrak{g}$, $i \in I$, are the Chevalley generators such that $\{e_i, h_i \mid i \in I\} \subset {\rm Lie}(B)$ and $\{f_i, h_i \mid i \in I\} \subset {\rm Lie}(B^-)$. For $i\in I$, let $\mathfrak{g}_i$ be the Lie subalgebra of $\mathfrak{g}$ generated by $e_i, f_i, h_i$, which is isomorphic to $\mathfrak{sl}_2(\mathbb{C})$ as a Lie algebra. For $m \in \z_{\ge 0}$, we denote by $\mathcal{B}^{(i)}(m)$ the crystal basis for the $(m+1)$-dimensional irreducible $\mathfrak{g}_i$-module with highest weight element $b_m$. We fix $i \in I$ and ${\bf c} \in (\z^{d_i})^\perp$ such that $\Psi_{\bf i} (\mathcal{B}(\lambda)) \cap ({\bf c} + \z^{d_i}) \neq \emptyset$. Recall that $\mu^{(i)} ({\bf c}) = (\mu_1 ^{(i)}({\bf c}), \ldots, \mu_{d_i} ^{(i)}({\bf c}))$, $\nu^{(i)} ({\bf c}) = (\nu_1 ^{(i)}({\bf c}), \ldots, \nu_{d_i} ^{(i)}({\bf c})) \in \z^{d_i}$ are uniquely determined by
\begin{align*}
&\Psi_{\bf i} (\mathcal{B}(\lambda)) \cap ({\bf c} + \z^{d_i}) = {\bf c} + \Pi_\z(\mu^{(i)} ({\bf c}), \nu^{(i)} ({\bf c})).
\end{align*} 
We define a bijective map \[\eta_i \colon \mathcal{B}(\lambda) \cap \Psi_{\bf i} ^{-1} ({\bf c} + \z^{d_i}) \xrightarrow{\sim} \mathcal{B}^{(i)}(\nu_1 ^{(i)}({\bf c}) - \mu_1 ^{(i)}({\bf c})) \otimes \cdots \otimes \mathcal{B}^{(i)}(\nu_{d_i} ^{(i)}({\bf c}) - \mu_{d_i} ^{(i)}({\bf c}))\] by \[\eta_i (b) \coloneqq \tilde{f}_i ^{a_1 ^{(i)} -\mu_1 ^{(i)}({\bf c})} b_{\nu_1 ^{(i)}({\bf c}) - \mu_1 ^{(i)}({\bf c})} \otimes \cdots \otimes \tilde{f}_i ^{a_{d_i} ^{(i)} -\mu_{d_i} ^{(i)}({\bf c})} b_{\nu_{d_i} ^{(i)}({\bf c}) - \mu_{d_i} ^{(i)}({\bf c})}\] when $\Psi_{\bf i} (b) = {\bf c} + (a^{(i)} _1, \ldots, a^{(i)} _{d_i})$ in ${\bf c} + \z^{d_i}$. 

\vspace{2mm}\begin{prop}
The map $\eta_i$ is an isomorphism of $\mathfrak{g}_i$-crystals.
\end{prop}

\begin{proof}
It suffices to prove that $\eta_i$ is compatible with the actions of $\tilde{e}_i$ and $\tilde{f}_i$. We show that $\eta_i (\tilde{e}_i b) = \tilde{e}_i \eta_i (b)$ for all $b \in \mathcal{B}(\lambda) \cap \Psi_{\bf i} ^{-1} ({\bf c} + \z^{d_i})$, where we set $\eta_i (0) \coloneqq 0$ if $\tilde{e}_i b = 0$; a proof of the compatibility with $\tilde{f}_i$ is similar. Let $b_{\rm high}$ (resp., $b_{\rm low}$) be the unique element in $\mathcal{B}(\lambda) \cap \Psi_{\bf i} ^{-1} ({\bf c} + \z^{d_i})$ such that $\Psi_{\bf i}(b_{\rm high}) = {\bf c} + \mu^{(i)}({\bf c})$ (resp., $\Psi_{\bf i}(b_{\rm low}) = {\bf c} + \nu^{(i)}({\bf c})$). Considering the weights of elements in the $\mathfrak{g}_i$-crystal $\mathcal{B}(\lambda) \cap \Psi_{\bf i} ^{-1} ({\bf c} + \z^{d_i})$, the standard representation theory of $\mathfrak{sl}_2(\mathbb{C})$ implies that $b_{\rm high}$ is the highest weight element in the $i$-string through $b_{\rm low}$. By the crystal structure on $\z_{\bf j} ^\infty \otimes R_\lambda$, this implies that 
\begin{equation}\label{eq:compatibility of longest}
\begin{aligned}
\eta_i (\tilde{e}_i ^k b_{\rm low}) = \tilde{e}_i ^k \eta_i (b_{\rm low}) 
\end{aligned}
\end{equation}
for all $k \in \z_{\ge 0}$. We set \[\{s_1 < \cdots < s_{d_i}\} \coloneqq \{1 \le s \le N \mid i_s = i\}.\] For $b \in \mathcal{B}(\lambda) \cap \Psi_{\bf i} ^{-1} ({\bf c} + \z^{d_i})$, define $\Upsilon_1 (b), \Upsilon_2 (b), \ldots, \Upsilon_{d_i + 1} (b)$ by
\begin{align*}
\widetilde{\Psi}_{\bf j} (b) &= \Upsilon_1 (b) \otimes \Upsilon_2 (b) \otimes \cdots \otimes \Upsilon_{d_i + 1} (b)\\
&\in (\z_{{\bf j}_{\ge N+1}} ^\infty \otimes \widetilde{\mathcal{B}}_{i_1} \otimes \cdots \otimes \widetilde{\mathcal{B}}_{i_{s_1}}) \otimes (\widetilde{\mathcal{B}}_{i_{s_1 + 1}} \otimes \cdots \otimes \widetilde{\mathcal{B}}_{i_{s_2}}) \otimes \cdots \otimes (\widetilde{\mathcal{B}}_{i_{s_{d_i} + 1}} \otimes \cdots \otimes \widetilde{\mathcal{B}}_{i_N} \otimes R_\lambda),
\end{align*}
and set $\Upsilon_{\le k} (b) \coloneqq \Upsilon_1 (b) \otimes \Upsilon_2 (b) \otimes \cdots \otimes \Upsilon_{k} (b)$ for $1 \le k \le d_i$. In addition, for $b \in \mathcal{B}^{(i)}(\nu_1 ^{(i)}({\bf c}) - \mu_1 ^{(i)}({\bf c})) \otimes \cdots \otimes \mathcal{B}^{(i)}(\nu_{d_i} ^{(i)}({\bf c}) - \mu_{d_i} ^{(i)}({\bf c}))$, we define $\Upsilon_1 (b), \Upsilon_2 (b), \ldots, \Upsilon_{d_i} (b)$ by
\begin{align*}
b &= \Upsilon_1 (b) \otimes \Upsilon_2 (b) \otimes \cdots \otimes \Upsilon_{d_i} (b)\\
&\in \mathcal{B}^{(i)}(\nu_1 ^{(i)}({\bf c}) - \mu_1 ^{(i)}({\bf c})) \otimes \mathcal{B}^{(i)}(\nu_2 ^{(i)}({\bf c}) - \mu_2 ^{(i)}({\bf c})) \otimes \cdots \otimes \mathcal{B}^{(i)}(\nu_{d_i} ^{(i)}({\bf c}) - \mu_{d_i} ^{(i)}({\bf c})),
\end{align*}
and set $\Upsilon_{\le k} (b) \coloneqq \Upsilon_1 (b) \otimes \Upsilon_2 (b) \otimes \cdots \otimes \Upsilon_{k} (b)$ for $1 \le k \le d_i$. By the tensor product rule for crystals, it suffices to prove that 
\begin{align*}
&\varepsilon_i (\Upsilon_{\le k} (b)) = \varepsilon_i (\Upsilon_{\le k} (\eta_i (b))),\ 1 \le k \le d_i,\\ 
&\varepsilon_i (\Upsilon_k (b)) -\varphi_i (\Upsilon_{\le k-1} (b)) = \varepsilon_i (\Upsilon_k (\eta_i(b))) -\varphi_i (\Upsilon_{\le k-1} (\eta_i(b))),\ 2 \le k \le d_i,
\end{align*} 
for $b \in \mathcal{B}(\lambda) \cap \Psi_{\bf i} ^{-1} ({\bf c} + \z^{d_i})$. We proceed by induction on $k$. 

If $k = 1$, then we take $b'$ in the $i$-string through $b_{\rm low}$ such that $\Upsilon_1 (b') = \Upsilon_1 (b)$; the existence of $b'$ follows by \eqref{eq:compatibility of longest}. Then, we deduce that 
\begin{align*}
\varepsilon_i (\Upsilon_{\le 1} (b)) &= \varepsilon_i (\Upsilon_1 (b'))\\
&= \varepsilon_i (\Upsilon_1 (\eta_i (b')))\quad({\rm by}\ \eqref{eq:compatibility of longest})\\
&= \varepsilon_i (\Upsilon_1 (\eta_i (b)))\quad({\rm by\ the\ definition\ of}\ \eta_i)\\
&= \varepsilon_i (\Upsilon_{\le 1} (\eta_i (b))).
\end{align*}

If $k \ge 2$, then we take $b''$ in the $i$-string through $b_{\rm low}$ such that \[\Upsilon_{\le k-1} (b'') = \Upsilon_{\le k-1} (b_{\rm low}),\ \Upsilon_k (b'') = \Upsilon_k (b);\] the existence of $b''$ follows by \eqref{eq:compatibility of longest}. Then, it follows that 
\begin{equation}\label{eq:tensor product compatibility 1}
\begin{aligned}
\varepsilon_i (\Upsilon_k (b)) -\varphi_i (\Upsilon_{\le k-1} (b)) =\ &\varepsilon_i (\Upsilon_k (b)) - \varphi_i (\Upsilon_{\le k-1} (b'')) + \varphi_i (\Upsilon_{\le k-1} (b'')) - \varphi_i (\Upsilon_{\le k-1} (b))\\
=\ &\varepsilon_i (\Upsilon_k (b'')) - \varphi_i (\Upsilon_{\le k-1} (b'')) + \varphi_i (\Upsilon_{\le k-1} (b_{\rm low})) - \varphi_i (\Upsilon_{\le k-1} (b))\\
=\ &\varepsilon_i (\Upsilon_k (b'')) - \varphi_i (\Upsilon_{\le k-1} (b'')) + \varepsilon_i (\Upsilon_{\le k-1} (b_{\rm low})) - \varepsilon_i (\Upsilon_{\le k-1} (b))\\ 
&+ \langle {\rm wt}(\Upsilon_{\le k-1} (b_{\rm low})), h_i \rangle - \langle {\rm wt}(\Upsilon_{\le k-1} (b)), h_i \rangle.
\end{aligned}
\end{equation}
Note that the following equality holds by \eqref{eq:compatibility of longest}:
\begin{equation}\label{eq:tensor product compatibility 2}
\begin{aligned}
\varepsilon_i (\Upsilon_k (b'')) - \varphi_i (\Upsilon_{\le k-1} (b'')) = \varepsilon_i (\Upsilon_k (\eta_i (b''))) - \varphi_i (\Upsilon_{\le k-1} (\eta_i (b''))). 
\end{aligned}
\end{equation}
In addition, we deduce by \eqref{eq:compatibility of longest} and by the induction hypothesis that 
\begin{equation}\label{eq:tensor product compatibility 3}
\begin{aligned}
\varepsilon_i (\Upsilon_{\le k-1} (b_{\rm low})) - \varepsilon_i (\Upsilon_{\le k-1} (b)) = \varepsilon_i (\Upsilon_{\le k-1} (\eta_i (b_{\rm low}))) - \varepsilon_i (\Upsilon_{\le k-1} (\eta_i (b))).
\end{aligned}
\end{equation}
If we write $\Psi_{\bf i} (b) = {\bf c} + (a^{(i)} _1, \ldots, a^{(i)} _{d_i})$ in ${\bf c} + \z^{d_i}$, then we have 
\begin{equation}\label{eq:tensor product compatibility 4}
\begin{aligned}
&\langle {\rm wt}(\Upsilon_{\le k-1} (b_{\rm low})), h_i \rangle - \langle {\rm wt}(\Upsilon_{\le k-1} (b)), h_i \rangle\\ 
=\ &2 \sum_{1 \le l \le k-1} (a^{(i)} _l - \nu^{(i)} _l ({\bf c}))\\
=\ &\langle {\rm wt}(\Upsilon_{\le k-1} (\eta_i (b_{\rm low}))), h_i \rangle - \langle {\rm wt}(\Upsilon_{\le k-1} (\eta_i (b))), h_i \rangle
\end{aligned}
\end{equation}
by the definition of $\eta_i$. By \eqref{eq:tensor product compatibility 1}--\eqref{eq:tensor product compatibility 4}, it follows that 
\begin{align*}
&\ \ \ \ \varepsilon_i (\Upsilon_k (b)) -\varphi_i (\Upsilon_{\le k-1} (b))\\ 
&= \varepsilon_i (\Upsilon_k (\eta_i (b''))) - \varphi_i (\Upsilon_{\le k-1} (\eta_i (b''))) + \varepsilon_i (\Upsilon_{\le k-1} (\eta_i (b_{\rm low}))) - \varepsilon_i (\Upsilon_{\le k-1} (\eta_i (b)))\\ 
&\ \ \ \ + \langle {\rm wt}(\Upsilon_{\le k-1} (\eta_i (b_{\rm low}))), h_i \rangle - \langle {\rm wt}(\Upsilon_{\le k-1} (\eta_i (b))), h_i \rangle\\
&= \varepsilon_i (\Upsilon_k (\eta_i (b''))) - \varphi_i (\Upsilon_{\le k-1} (\eta_i (b''))) + \varphi_i (\Upsilon_{\le k-1} (\eta_i (b_{\rm low}))) - \varphi_i (\Upsilon_{\le k-1} (\eta_i (b)))\\
&= \varepsilon_i (\Upsilon_k (\eta_i (b))) - \varphi_i (\Upsilon_{\le k-1} (\eta_i (b''))) + \varphi_i (\Upsilon_{\le k-1} (\eta_i (b''))) - \varphi_i (\Upsilon_{\le k-1} (\eta_i (b)))\\
&({\rm by\ the\ definition\ of}\ \eta_i)\\
&= \varepsilon_i (\Upsilon_k (\eta_i (b))) - \varphi_i (\Upsilon_{\le k-1} (\eta_i (b))),
\end{align*}
and hence that
\begin{align*}
\varepsilon_i (\Upsilon_{\le k} (b)) &= \max\{\varepsilon_i (\Upsilon_{\le k-1} (b)), \varepsilon_i (\Upsilon_{\le k-1} (b)) + \varepsilon_i (\Upsilon_k (b)) -\varphi_i (\Upsilon_{\le k-1} (b))\}\\
&({\rm by\ the\ tensor\ product\ rule\ for\ crystals})\\
&= \max\{\varepsilon_i (\Upsilon_{\le k-1} (\eta_i (b))), \varepsilon_i (\Upsilon_{\le k-1} (\eta_i (b))) + \varepsilon_i (\Upsilon_k (\eta_i (b))) -\varphi_i (\Upsilon_{\le k-1} (\eta_i (b)))\}\\
&= \varepsilon_i (\Upsilon_{\le k} (\eta_i (b))). 
\end{align*}
This proves the proposition.
\end{proof}

\section{Geometric applications}

In this section, we discuss toric degenerations arising from Nakashima-Zelevinsky polytopes by the theory of Newton-Okounkov bodies \cite{And}. We start with recalling the main result of \cite{FN}, which states that $\Delta_{\bf i} (\lambda)$ is identical to the Newton-Okounkov body of the full flag variety $G/B$ associated with a specific valuation. For $\lambda \in P_+$, we define a line bundle $\mathcal{L}_\lambda$ on $G/B$ by \[\mathcal{L}_\lambda \coloneqq (G \times \mathbb{C})/B,\] where $B$ acts on $G \times \mathbb{C}$ on the right as follows: \[(g, c) \cdot b = (g b, \lambda(b) c)\] for $g \in G$, $c \in \mathbb{C}$, and $b \in B$. Take a reduced word ${\bf i} = (i_1, \ldots, i_N) \in I^N$ for the longest element $w_0 \in W$. We see by \cite[Ch.\ I\hspace{-.1em}I.13]{Jan} that the morphism \[\c^N \rightarrow G/B,\ (t_1, \ldots, t_N) \mapsto \exp(t_1 f_{i_1}) \cdots \exp(t_N f_{i_N}) \bmod B,\] is birational. Hence the function field $\mathbb{C}(G/B)$ is identified with the rational function field $\mathbb{C}(t_1, \ldots, t_N)$. 

\vspace{2mm}\begin{defi}\label{d:valuations}\normalfont
We define a lexicographic order $\prec$ on $\mathbb{Z}^N$ as follows: $(a_1, \ldots, a_N) \prec (a_1 ^\prime, \ldots, a_N ^\prime)$ if and only if there exists $1 \le k \le N$ such that $a_N = a_N ^\prime, \ldots, a_{k+1} = a_{k+1} ^\prime$, $a_k < a_k ^\prime$. The lexicographic order $\prec$ on $\mathbb{Z}^N$ induces a total order (denoted by the same symbol $\prec$) on the set of monomials in the polynomial ring $\mathbb{C}[t_1, \ldots, t_N]$ as follows: $t_1 ^{a_1} \cdots t_N ^{a_N} \prec t_1 ^{a_1 ^\prime} \cdots t_N ^{a_N ^\prime}$ if and only if $(a_1, \ldots, a_N) \prec (a_1 ^\prime, \ldots, a_N ^\prime)$. Let us define a valuation $v_{{\bf i}, \prec} ^{\rm high} \colon \mathbb{C}(G/B) \setminus \{0\} \rightarrow \mathbb{Z}^N$ by $v_{{\bf i}, \prec} ^{\rm high} (f/g) \coloneqq v_{{\bf i}, \prec} ^{\rm high} (f) - v_{{\bf i}, \prec} ^{\rm high} (g)$ for $f, g \in \mathbb{C}[t_1, \ldots, t_N] \setminus \{0\}$, and by 
\begin{align*}
&v_{{\bf i}, \prec} ^{\rm high} (f) \coloneqq -(a_1, \ldots, a_N)\ {\rm for}\ f = c t_1 ^{a_1} \cdots t_N ^{a_N} + ({\rm lower\ terms}) \in \mathbb{C}[t_1, \ldots, t_N] \setminus \{0\},
\end{align*}
where $c \in \mathbb{C} \setminus \{0\}$, and we mean by ``lower terms'' a linear combination of monomials smaller than $t_1 ^{a_1} \cdots t_N ^{a_N}$ with respect to the total order $\prec$. 
\end{defi}

\vspace{2mm}\begin{defi}[{see \cite[Sect.\ 1.2]{Kav} and \cite[Definition 1.10]{KK2}}]\normalfont\label{d:Newton-Okounkov polytopes}
Let ${\bf i} \in I^N$ be a reduced word for $w_0$, and $\lambda \in P_+$. Take a nonzero section $\tau \in H^0(G/B, \mathcal{L}_\lambda)$. We define a subset $S(G/B, \mathcal{L}_\lambda, v_{{\bf i}, \prec} ^{\rm high}, \tau) \subset \mathbb{Z}_{>0} \times \mathbb{Z}^N$ by \[S(G/B, \mathcal{L}_\lambda, v_{{\bf i}, \prec} ^{\rm high}, \tau) \coloneqq \bigcup_{k>0} \{(k, v_{{\bf i}, \prec} ^{\rm high}(\sigma/\tau^k)) \mid \sigma \in H^0(G/B, \mathcal{L}_\lambda ^{\otimes k}) \setminus \{0\}\},\] and denote by $C(G/B, \mathcal{L}_\lambda, v_{{\bf i}, \prec} ^{\rm high}, \tau) \subset \mathbb{R}_{\ge 0} \times \mathbb{R}^N$ the smallest real closed cone containing $S(G/B, \mathcal{L}_\lambda, v_{{\bf i}, \prec} ^{\rm high}, \tau)$. Let us define a subset $\Delta(G/B, \mathcal{L}_\lambda, v_{{\bf i}, \prec} ^{\rm high}, \tau) \subset \mathbb{R}^N$ by \[\Delta(G/B, \mathcal{L}_\lambda, v_{{\bf i}, \prec} ^{\rm high}, \tau) \coloneqq \{{\bf a} \in \mathbb{R}^N \mid (1, {\bf a}) \in C(G/B, \mathcal{L}_\lambda, v_{{\bf i}, \prec} ^{\rm high}, \tau)\};\]
this is called the \emph{Newton-Okounkov body} of $G/B$ associated with $\mathcal{L}_\lambda$, $v_{{\bf i}, \prec} ^{\rm high}$, and $\tau$.
\end{defi}\vspace{2mm}

We define an $\r$-linear automorphism $\omega \colon \mathbb{R} \times \mathbb{R}^N \xrightarrow{\sim} \mathbb{R} \times \mathbb{R}^N$ by $\omega(k, {\bf a}) \coloneqq (k, -{\bf a})$. 

\vspace{2mm}\begin{thm}[{see \cite[Sect.\ 4]{FN}}]\label{t:NOBY polyhedral realizations}
Let ${\bf i} \in I^N$ be a reduced word for $w_0$, and $\lambda \in P_+$. Then, there exists a nonzero section $\tau_\lambda \in H^0(G/B, \mathcal{L}_\lambda)$ such that the following equalities hold$:$
\begin{align*}
&\mathcal{S}_{\bf i} (\lambda) = \omega(S(G/B, \mathcal{L}_\lambda, v_{{\bf i}, \prec} ^{\rm high}, \tau_\lambda)),\ \mathcal{C}_{\bf i} (\lambda) = \omega(C(G/B, \mathcal{L}_\lambda, v_{{\bf i}, \prec} ^{\rm high}, \tau_\lambda)),\ {\it and}\\
&\Delta_{\bf i} (\lambda) = -\Delta(G/B, \mathcal{L}_\lambda, v_{{\bf i}, \prec} ^{\rm high}, \tau_\lambda).
\end{align*}
\end{thm}\vspace{2mm}

\begin{rem}\normalfont
The author and Oya \cite{FO} proved that $\Delta_{\bf i} (\lambda)$ is also identical to the Newton-Okounkov body of $G/B$ associated with a geometrically natural valuation, which is given by counting the orders of zeros along a specific sequence of Schubert subvarieties.
\end{rem}\vspace{2mm}

We say that $G/B$ admits a \emph{flat degeneration} to a variety $X$ if there exists a flat morphism 
\[\pi \colon \mathfrak{X} \rightarrow {\rm Spec}(\c[t])\] 
of schemes such that the scheme-theoretic fiber $\pi^{-1}(t)$ (resp., $\pi^{-1}(0)$) over a closed point $t \in \c \setminus \{0\}$ (resp., the origin $0 \in \c$) is isomorphic to $G/B$ (resp., $X$). By Theorem \ref{t:NOBY polyhedral realizations} and \cite[Theorem 1]{And} (see also \cite[Corollary 3.14]{HK}), there exists a flat degeneration of $G/B$ to ${\rm Proj} (\c[\mathcal{S}_{\bf i} (\lambda)])$, where the $\z_{>0}$-grading of $\mathcal{S}_{\bf i} (\lambda)$ induces a $\z_{\ge 0}$-grading of $\c[\mathcal{S}_{\bf i} (\lambda)]$. By Proposition \ref{p:convexity} (1) and \cite[Theorem 1.3.5]{CLS}, we see that ${\rm Proj} (\c[\mathcal{S}_{\bf i} (\lambda)])$ is normal; hence it is identical to the normal toric variety $X(\Delta_{\bf i} (\lambda))$ associated with the rational convex polytope $\Delta_{\bf i} (\lambda)$. Thus, we obtain the following.

\vspace{2mm}\begin{thm}\label{t:toric deg}
There exists a flat degeneration of $G/B$ to the normal toric variety $X(\Delta_{\bf i} (\lambda))$ associated with the Nakashima-Zelevinsky polytope $\Delta_{\bf i} (\lambda)$.
\end{thm}\vspace{2mm}

We apply Alexeev-Brion's argument \cite{AB} to this flat degeneration. 

\vspace{2mm}\begin{defi}\normalfont
Let ${\bf i} \in I^N$ be a reduced word for $w_0$, and write $P_\r \coloneqq P \otimes_\z \r$. Define a subset $\mathcal{S}_{\bf i} \subset P_+ \times \z^N$ by \[\mathcal{S}_{\bf i} \coloneqq \bigcup_{\lambda \in P_+} \{(\lambda, \Psi_{\bf i}(b)) \mid b \in \mathcal{B}(\lambda)\},\] and denote by $\mathcal{C}_{\bf i} \subset P_\r \times \r^N$ the smallest real closed cone containing $\mathcal{S}_{\bf i}$. 
\end{defi}\vspace{2mm}

In a way similar to the proof of \cite[Corollaries 2.18 and 4.3]{FN}, we deduce the following.

\vspace{2mm}\begin{prop}
Let ${\bf i} \in I^N$ be a reduced word for $w_0$. Then, the real closed cone $\mathcal{C}_{\bf i}$ is a rational convex polyhedral cone, and the equality $\mathcal{S}_{\bf i} = \mathcal{C}_{\bf i} \cap (P_+ \times \z^N)$ holds.
\end{prop}\vspace{2mm}

Let $\{\varpi_i \mid i \in I\} \subset P_+$ be the set of fundamental weights, and $P_{\r, +} \subset P_\r$ the closure of the fundamental Weyl chamber with respect to the Euclidean topology, that is, \[P_{\r, +} \coloneqq \sum_{i \in I} \r_{\ge 0} \varpi_i.\] Denote by $\pi_1 \colon P_\r \times \r^N \rightarrow P_\r$ the first projection, which maps the rational convex polyhedral cone $\mathcal{C}_{\bf i}$ onto $P_{\r, +}$. Then, for $\lambda \in P_+$, the Nakashima-Zelevinsky polytope $\Delta_{\bf i}(\lambda)$ is identical to the fiber $\mathcal{C}_{\bf i} \cap \pi_1 ^{-1} (\lambda)$. Imitating \cite[Definition 4.1]{AB}, we define a fan $\Sigma_{\bf i}$ from $\mathcal{C}_{\bf i}$. For $\lambda \in P_{\r, +}$, we set 
\begin{align*}
&F_\lambda \coloneqq \{{\rm faces}\ \tau\ {\rm of}\ \mathcal{C}_{\bf i} \mid \lambda \in \pi_1 (\tau^0)\},\ {\rm and}\\
&\sigma_\lambda ^0 \coloneqq \bigcap_{\tau \in F_\lambda} \pi_1 (\tau^0),
\end{align*}
where $\tau^0$ is the relative interior of $\tau$. Denote by $\sigma_\lambda$ the closure of $\sigma_\lambda ^0$ in $P_\r$ with respect to the Euclidean topology. Then, a fan $\Sigma_{\bf i}$ with support $P_{\r, +}$ is defined to be \[\Sigma_{\bf i} \coloneqq \{\sigma_\lambda \mid \lambda \in P_{\r, +}\};\] the fan $\Sigma_{\bf i}$ is said to be \emph{trivial} if it consists only of the faces of $P_{\r, +}$. Let $P_{++} \subset P_+$ denote the set of regular dominant integral weights. For $\lambda \in P_{++}$, the line bundle $\mathcal{L}_\lambda$ on $G/B$ is very ample (see, for instance, \cite[Sect.~I\hspace{-.1em}I.8.5]{Jan}); hence we see by \cite[Corollary 3.2]{KK2} that the real dimension of $\Delta_{\bf i}(\lambda)$ equals $N$. In a way similar to the argument in \cite{AB}, we obtain the following.

\vspace{2mm}\begin{prop}[{cf.\ \cite[Lemma 4.2 and Corollary 4.3]{AB}}]\label{p:structure of fan}
Let ${\bf i} \in I^N$ be a reduced word for $w_0$.
\begin{enumerate}
\item[{\rm (1)}] Two weights $\lambda, \mu \in P_+$ lie in the same cone of $\Sigma_{\bf i}$ if and only if $\Delta_{\bf i}(\lambda + \mu)$ is the Minkowski sum of $\Delta_{\bf i}(\lambda)$ and $\Delta_{\bf i}(\mu)$.
\item[{\rm (2)}] If the fan $\Sigma_{\bf i}$ is trivial, then the polytopes $\Delta_{\bf i}(\lambda)$, $\lambda \in P_{++}$, have the same normal fan. 
\end{enumerate}
\end{prop}\vspace{2mm}

The following is an immediate consequence of Theorem \ref{t:corollary 1} and Proposition \ref{p:structure of fan}.

\vspace{2mm}\begin{cor}\label{c:geometric corollary1}
If $\Delta_{\bf i}(\lambda)$ is a parapolytope for all $\lambda \in P_+$, then the toric varieties $X(\Delta_{\bf i}(\lambda))$, $\lambda \in P_{++}$, are all identical.
\end{cor}\vspace{2mm}

We say that $X(\Delta_{\bf i}(\lambda))$ is \emph{Gorenstein Fano} if the anti-canonical class $-K_{X(\Delta_{\bf i}(\lambda))}$ is Cartier and ample (see \cite[Sect.\ 8.3]{CLS}). Let $\mathcal{O}(K_{G/B})$ denote the canonical bundle of $G/B$. By \cite[Proposition 2.2.7 (ii)]{Bri}, we have $\mathcal{O}(K_{G/B}) \simeq \mathcal{L}_{-2\rho}$, where $\rho \in P_{++}$ is the half sum of the positive roots. By the argument in the proof of \cite[Proposition 2.4]{AB} (see also \cite[Theorem 3.8]{AB}), the anti-canonical sheaf $\mathcal{O}(-K_{X(\Delta_{\bf i}(2\rho))})$ is the limit of $\mathcal{L}_{2\rho} \simeq \mathcal{O}(-K_{G/B})$ under the flat degeneration of $G/B$ to $X(\Delta_{\bf i}(2\rho))$ in Theorem \ref{t:toric deg}. Hence we obtain the following by Theorem \ref{t:main result} (1).

\vspace{2mm}\begin{cor}\label{c:geometric corollary2}
If $\Delta_{\bf i}(2 \rho)$ is a parapolytope, then the toric variety $X(\Delta_{\bf i}(2\rho))$ is Gorenstein Fano, that is, $\Delta_{\bf i}(2\rho)$ is reflexive.
\end{cor}\vspace{2mm}

By Corollaries \ref{c:geometric corollary1} and \ref{c:geometric corollary2}, we obtain Corollary 4 in Introduction.

\end{document}